\documentclass[a4paper, 11pt, reqno]{amsart}
\usepackage{amssymb,amsmath}
\usepackage{color}
\usepackage{graphicx}

\usepackage[colorlinks=true, citecolor=blue]{hyperref}

\usepackage{tikz}
\usetikzlibrary{arrows,automata,backgrounds,calendar,mindmap,trees}
\usepackage{xcolor}

\usepackage[margin=1in]{geometry}

\newcommand{\Ocyt}{\Omega_{\mathrm{cyt}}}
\newcommand{\Onuc}{\Omega_{\mathrm{nuc}}}

\newcommand{\ract}{r_{\mathrm{act}}}
\newcommand{\rimp}{r_{\mathrm{imp}}}
\newcommand{\rimpp}{r_{\mathrm{imp2}}}
\newcommand{\rexp}{r_{\mathrm{exp}}}
\newcommand{\rdelay}{r_{\mathrm{delay}}}

\newcommand{\ukhong}{u_{0,\infty}}
\newcommand{\umot}{u_{1,\infty}}
\newcommand{\uhai}{u_{2,\infty}}
\newcommand{\uba}{u_{3,\infty}}
\newcommand{\ubon}{u_{4,\infty}}
\newcommand{\unam}{u_{5,\infty}}
\newcommand{\usau}{u_{6,\infty}}
\newcommand{\ubay}{u_{7,\infty}}

\newcommand{\uu}{\mathbf{u}}
\newcommand{\vv}{\mathbf{w}}
\newcommand{\uinf}{\mathbf{u}_{\infty}}

\newcommand{\E}{\mathcal{E}}
\newcommand{\D}{\mathcal{D}}

\newcommand{\intO}{\int_{\Omega}}
\newcommand{\intG}{\int_{\Gamma}}
\newcommand{\intGG}{\int_{\Gamma_2}}
\newcommand{\Li}{L_{\infty}}
\newcommand{\li}{\ell_{\infty}}
\renewcommand{\Pi}{P_{\infty}}
\renewcommand{\pi}{p_{\infty}}

\newcommand{\oU}{\overline{U}}
\newcommand{\oV}{\overline{V}}
\newcommand{\ou}{\overline{u}}
\newcommand{\ov}{\overline{v}}

\newcommand{\cc}{\boldsymbol{c}}

\begin{document}
\title[Volume-surface reaction-diffusion systems arising from cell biology]
{Entropy methods and convergence to equilibrium for volume-surface reaction-diffusion systems}

\author[K. Fellner, B.Q. Tang]{Klemens Fellner, Bao Quoc Tang}

\address{Klemens Fellner \hfill\break
Institute of Mathematics and Scientific Computing, University of Graz, Heinrichstrasse 36, 8010 Graz, Austria}
\email{klemens.fellner@uni-graz.at}

\address{Bao Quoc Tang \hfill\break
Institute of Mathematics and Scientific Computing, University of Graz, Heinrichstrasse 36, 8010 Graz, Austria}
%$^{2}$ Faculty of Applied Mathematics and Informatics, Hanoi University of Sience and Technology, 1 Dai Co Viet, Hai Ba Trung, Hanoi, Vietnam}
\email{quoc.tang@uni-graz.at} 


\begin{abstract}
We consider two volume-surface reaction-diffusion systems arising from cell biology. The first system describes the localisation of the protein Lgl in the asymmetric division of Drosophila SOP stem cells, while the second system models the JAK2/STAT5 signalling pathway. Both model systems have in common that i) different species are located in \emph{different spatial compartments}, ii) the involved chemical reaction kinetics between the species satisfies a \emph{complex balance condition} and iii) that the associated complex balance equilibrium is \emph{spatially inhomogeneous}. 
	By using recent advances on the entropy method for complex balanced reaction-diffusion systems, we show for both systems exponential convergence to the equilibrium with constants and rates, which can be explicitly estimated.
\end{abstract}

\maketitle

\numberwithin{equation}{section}
\newtheorem{theorem}{Theorem}[section]
\newtheorem{lemma}[theorem]{Lemma}
\newtheorem{proposition}[theorem]{Proposition}
\newtheorem{definition}{Definition}[section]
\newtheorem{remark}{Remark}[section]
\newtheorem{corollary}{Corollary}[section]

\section{Introduction and main results}

%{\bf Outline}
%\begin{itemize}
%	\item Volume-surface systems: what are they? where do they come from?
%	
%	\item (Bio-)Chemical reaction network: a little history? challenges? complex balanced condition?
%	
%	\item Aims and novelties of the paper?
%\end{itemize}

Various physical, chemical, biological or ecological systems involve processes in different spatial compartments.
A particular important example is given by considering quantities on a domain and on its surrounding boundary. Reactions taking place in these situations result in a class of PDE models called \emph{Volume-Surface Reaction-Diffusion} systems (hereafter, we will use the abbreviation {VSRD systems}). The intrinsic volume-surface coupling of VSRD systems introduces new difficulties in both the mathematical and the numerical analysis compared to classical reaction-diffusion systems supported on only one spatial domain. 

Recently, a rapidly increasing amount of attention has been devoted to the mathematical theory of VSRD systems arising from such different applications as fluid mechanics \cite{GM13, Mie13}, ecology \cite{BCRR15,BCR13,BRR13a}, crystal growth \cite{KD01} or, especially, cell biology, see e.g. \cite{FLT,FRT16,FNR13,HT16}. %The studied questions including, but not limited to, well-posedness, travelling waves, Turing instability or symmetry breaking, etc. 
\smallskip

This paper aims to investigate the large time behaviour of two particular linear VSRD systems arising from two different application backgrounds in cell biology: the first being a model for the localisation of the key-protein Lgl during the asymmetric stem cell division of SOP precursor cells in Drosophila (see e.g. \cite{FRT16}), the second one being a model on the so-called JAK2/STAT5 signalling pathway, see \cite{FNR13}.
In both examples, the cell cytoplasm constitutes the volume domain and the surrounding cell-membrane/cortex constitutes the surrounding boundary. In the JAK2/STAT5 model, also the volume of cell nucleus and its boundary are considered. 
Moreover, reactions occur within and between the volume- and surface-compartments, which do not satisfy the so called detailed balance condition, but the more general  
\emph{complex balance condition}, see e.g. \cite{HJ72}. 

The main results of this paper prove exponential convergence to the complex balance equilibrium with explicitly computable rates by using the so-called {\it entropy method}. Moreover, we extend the entropy method to apply to spatially inhomogeneous equilibria. 
\textcolor{black}{This paper may serve as a proof of concept for the applicability of the entropy method to a wide class of VSRD systems including mixed ODE/PDE systems. The presented proofs, however, rely on positive lower and upper bounds of the equilibria, which are difficult to obtain for general systems.}
\smallskip

The entropy method applied in this work has recently become a very powerful tool in proving exponential convergence to equilibrium with explicit rates for reaction-diffusion systems, but mostly under the assumption of the detailed balance condition and on a single spatial domain, see e.g. \cite{DF06, DF08,FT,GGH96, MHM15}. 
The entropy method for complex balance reaction-diffusion networks was so far only considered on a single domain for linear, respectively, nonlinear systems in \cite{FPT,DFT}. 

The current paper constitutes the first results of the entropy method for volume-surface systems with spatially inhomogeneous complex balance equilibria.   
In the next two sections, we detail the considered VSRD systems and state the main results.

%{\color{red}Note that for volume-surface reaction-diffusion systems, the complex balance condition is {\it different} from the normal one for volume-only reaction-diffusion systems.}

%  ******************************************************* %
\subsection*{A VSRD model for the localisation of Lgl during asymmetric stem cell division}
In stem cells undergoing asymmetric cell division, particular proteins (so-called cell-fate determinants) are localised at the cortex of only one of the two daughter cells during mitosis. These cell-fate determinants trigger in the following the differentiation of one daughter cell into specific tissue while the other daughter cell remains a stem cell. 
         
In Drosophila, SOP precursor stem cells provide a well-studied biological example model of asymmetric stem cell division, see e.g. \cite{BMK,MEBWK,WNK} and the references therein. 
In particular, asymmetric cell division of SOP cells is driven by the asymmetric localisation of the key protein Lgl (Lethal giant larvae), which exists in two conformational states: a non-phosphorylated form which regulates the localisation of the cell-fate-determinants in the membrane of one daughter cell, and a phosphorylated form which is inactive.

The asymmetric localisation of Lgl during mitosis is the result of the activation of the kinase aPKC, which phosphorylates Lgl (as part of a highly evolutionary conserved protein complex) only on a subpart of the cortex, as well as the results of the weakly reversible reaction/sorption dynamics of the two conformations of Lgl between cortex and cytoplasm. In particular, it is the  irreversible release of phosphorylated Lgl from the cortex, which initiates the asymmetric localisation of Lgl upon the activation of aPKC.
\smallskip
%While the asymmetric localisation of Lgl is essential for the asymmetric cell division of SOP cells, the subsequent biological machinery, where the asymmetric localisation of Lgl leads to the asymmetric localisation allocation of the adaptor protein Pon (Partner of Numb) and the cell-fate determinate Numb is currently not well enough understood to be considered here.

Let $\Omega \subset \mathbb R^n$, $n =2,3$ describe the cell cytoplasm as a connected, bounded domain with sufficiently smooth boundary $\partial\Omega$ (e.g. $\partial\Omega \in C^{2+\epsilon}$, $\epsilon>0$). Denote by $\Gamma = \partial\Omega$ the surrounding cell cortex. Moreover, the cell cortex is divided into two disjoint, connected subsets $\Gamma_1$ and $\Gamma_2$ with $\Gamma = \Gamma_1 \cup \Gamma_2$ and where $\Gamma_2$ is the active part of the cell cortex, where phosphorylation takes place. 

Concerning the various conformations of Lgl, we denote by $L$ and $P$ the concentrations of non-phosphorylated Lgl and phosphorylated Lgl within the volume domain $\Omega$. Moreover, the concentrations of non-phosphorylated Lgl and phosphorylated Lgl on the cell cortex are denoted by $\ell$ and $p$, respectively. Note that $\ell$ is supported on $\Gamma$ while $p$ is supported only on the sub-domain $\Gamma_2$, since phosphorylation only occurs on $\Gamma_2$. 

Schematically, we consider the following reactions between the four different conformations of Lgl with positive reaction rate constants $\alpha, \beta, \lambda, \gamma, \sigma$ and $\xi$. 
\begin{figure}[htp]
%\sidecaption
%\begin{center}
\scalebox{1}[1]{
\begin{tikzpicture}
\node (a) {$L$} node (b) at (1.7,0) {$P$} node (c) at (0,-1.7) {$\ell$} node (d) at (1.7,-1.7)  {$p$};
\draw[arrows=->] ([xshift =0.5mm,yshift=0.5mm]a.east) -- node [below] {\scalebox{.8}[.8]{$\alpha$}} ([xshift =-0.5mm,yshift=0.5mm]b.west);
\draw[arrows=->] ([xshift =-0.5mm,yshift=-0.5mm]b.west) -- node [above] {\scalebox{.8}[.8]{$\beta$}} ([xshift =0.5mm,yshift=-0.5mm]a.east);
\draw[arrows=->] ([xshift =0.5mm,yshift=0.5mm]c.north) -- node [right] {\scalebox{.8}[.8]{$\gamma$}} ([xshift =0.5mm,yshift=-0.5mm]a.south);
\draw[arrows=->] ([xshift =-0.5mm,yshift=-0.5mm]a.south) -- node [left] {\scalebox{.8}[.8]{$\lambda$}} ([xshift =-0.5mm,yshift=0.5mm]c.north);
\draw[arrows=->] ([xshift =0.5mm]c.east) -- node [below] {\scalebox{.8}[.8]{$\sigma$}} ([xshift=-0.5mm]d.west);
\draw[arrows=->] ([yshift =0.5mm]d.north) -- node [right] {\scalebox{.8}[.8]{$\xi$}} ([yshift=-0.5mm]b.south);
\end{tikzpicture}\hspace{2.5cm}} 
%\end{center}
\caption{The reaction dynamics between $L, P, l$ and $p$}\label{LGLmodel}
\end{figure}
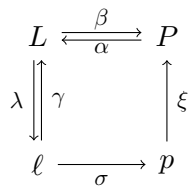

Moreover, we assume positive diffusion coefficients of $L, P$ on $\Omega$, of $\ell$ on $\Gamma$ and of $p$ on $\Gamma_2$, i.e. $d_L, d_P, d_{\ell}, d_p>0$, respectively. Then, applying the mass action law, the resulted VSRD system consists of two volume equations
\begin{subequations}\label{Lgl_system}
	\begin{equation}
		\begin{cases}
			L_t - d_L\Delta L = -\beta L + \alpha P, &x\in\Omega,\; t>0,\\
			P_t - d_P\Delta P = \beta L - \alpha P, &x\in\Omega,\; t>0,
		\end{cases}
	\end{equation}
	\text{and two surface equations}
	\begin{equation}
		\begin{cases}
			\ell_t - d_{\ell}\Delta_{\Gamma}\ell = \lambda L - (\gamma + \sigma\chi_{\Gamma_2})\ell, &x\in\Gamma, \; t>0,\\
			p_t - d_p\Delta_{\Gamma_2}p = \sigma\ell - \xi p, &x\in\Gamma_2, \; t>0,\\
			\quad \partial_{\nu_2}p = 0, &x\in\partial\Gamma_2, \; t>0,\\
		\end{cases}
	\end{equation}
	\text{which are connected via mixed Neumann/Robin boundary conditions}
	\begin{equation}
		\begin{cases}
			\quad d_L\partial_{\nu}L = -\lambda L + \gamma \ell, &x\in\Gamma,\; t>0,\\
			\quad d_P\partial_{\nu}P = \chi_{\Gamma_2}\xi p, &x\in\Gamma,\; t>0,
		\end{cases}
	\end{equation}
	\text{and subject to nonnegative initial data}
	\begin{equation}
		\begin{cases}
			L(x,0) = L_0(x),\quad P(x,0) = P_0(x),\qquad x\in\Omega,\\
			\ell(x,0) = \ell_0(x), \quad x\in \Gamma,\qquad	p(x,0) = p_0(x), \quad x\in\Gamma_2,
		\end{cases}
	\end{equation}
\end{subequations}
where $\nu$ and $\nu_2$ are the outward unit normal vectors of $\Gamma$ and $\partial\Gamma_2$, respectively, and  $\Delta_{\Gamma}$ and $\Delta_{\Gamma_2}$ are Laplace-Beltrami operators on $\Gamma$ and $\Gamma_2$, respectively. Moreover, $\chi_{\Gamma_2}$ denotes the characteristic function of the boundary part $\Gamma_2$.

Note that the above system \eqref{Lgl_system} conserves the total mass of Lgl, which is expressed in the following conservation law:
\begin{multline}
         \int_{\Omega} (L(t,x)+P(t,x)) \, dx + \int_{\Gamma} \ell(t,x)\,dS + \int_{\Gamma_2} p(t,x)\,dS \\
         = \int_{\Omega} (L_0(x)+P_0(x)) \, dx + \int_{\Gamma} \ell_0(x)\,dS + \int_{\Gamma_2} p_0(x)\,dS>0,\qquad \forall t>0 \label{cons}
         \end{multline}

 \medskip

Concerning the existence of global weak solutions of \eqref{Lgl_system}, we refer to \cite{FRT16}, where also 
the quasi-steady-state approximation in the limit $\xi \rightarrow +\infty$ and numerical simulations were carried out. 

The first main result of this paper is stated in the following theorem.

\begin{theorem}[Exponential convergence to equilibrium of Lgl system \eqref{Lgl_system}]\label{mainresult1}\hfill\\
	Assume that $\Omega\subset \mathbb R^n$, $n=2,3$ is a bounded domain with sufficiently smooth boundary $\Gamma = \partial\Omega$. Moreover, $\Gamma = \Gamma_1 \cup \Gamma_2$ is the union of two disjoint, connected subsets and $\Gamma_2$ has a smooth boundary $\partial\Gamma_2$.
	
	Then, for any positive initial mass $M>0$, the system \eqref{Lgl_system} possesses a unique positive equilibrium $(\Li, \Pi, \li, \pi)$ satisfying the mass conservation law
	\begin{equation*}
		\int_{\Omega}(\Li(x) + \Pi(x))dx + \int_{\Gamma}\li(x)dS + \int_{\Gamma_2}\pi(x)dS = M.
	\end{equation*}
Moreover, $\Li \in C(\overline{\Omega})\cap H^2(\Omega)$, $\Pi \in L^{\infty}({\Omega})\cap H^{3/2}(\Omega)$, $\li\in C({\Gamma})\cap H^2(\Gamma_2)$ and $\pi\in C(\overline{\Gamma_2})\cap H^2(\Gamma_2)$, and there are $0<a\leq A <+\infty$ such that
\begin{equation*}
	a\leq \Li(x), \Pi(x), \li(x), \pi(x) \leq A
\end{equation*}
where the bounds hold for $x$ in $\overline{\Omega}$, a.e. in $\Omega$, in $\Gamma$ and in $\overline{\Gamma_2}$ respectively.
%we have the following upper- and lower bounds 
%	\begin{equation*}
%			0 < a \leq \Li(x)\leq A \quad \text{ for all } x\in\overline{\Omega},
%	\end{equation*}
%	\begin{equation*}
%			0 < a \leq \Pi(x)\leq A \quad \text{ for almost all } x\in\Omega,
%	\end{equation*}
%	\begin{equation*}
%		0 < a \leq \li(x) \leq A \qquad \text{ for all } x\in{\Gamma}
%	\end{equation*}
%	%and
%	\begin{equation*}
%		0 < a \leq \pi(x) \leq A \qquad \text{ for all } x\in\overline{\Gamma_2}
%	\end{equation*}
%	for some constants $0<a\leq A <+\infty$.
	
	Finally, every global weak solution to \eqref{Lgl_system}  with positive initial mass $M$ (as constructed in \cite{FRT16}) converges exponentially to $(\Li, \Pi, \li, \pi)$ in the following sense
	\begin{equation*}
		\int_{\Omega}\frac{|L(t) - \Li|^2}{\Li} + \int_{\Omega}\frac{|P(t) - \Pi|^2}{\Pi} + \int_{\Gamma}\frac{|\ell(t) - \li|^2}{\li} + \int_{\Gamma_2}\frac{|p(t) - \pi|^2}{\pi} \leq C_0e^{-\lambda_0 t}
	\end{equation*}
	for all $t>0$, where $\lambda_0$ is as in Lemma \ref{EEDE_Lgl} and $C_0$ and $\lambda_0$ can be computed explicitly. 
\end{theorem}

%  ******************************************************* %
\subsection*{A PDE/ODE system modelling the JAK2/STAT5 signalling pathway}

The communication between cells in multicellular organisms is often mediated by signalling molecules secreted to the extracellular space, which then bind to cell surface receptors, see \cite{FNR13}. 
However, the modalities of the transport from the site of signal transducer and activator of transcription (STAT) phosphorylation at the plasma membrane to the site of action in the nucleus is still unclear. In \cite{FNR13}, Friedmann, Neumann and Rannacher  introduced a mixed PDE/ODE model system to analyse the influence of the cell shape on the regulatory response to the activated pathway.

By following the notations in \cite{FNR13}, we denote by $u_0$ and $u_1$ the unphosporylated and phosphorylated STAT5 in the cytoplasm, while $u_2$ and $u_3$ denote the unphosphorylated and phosphorylated STAT5 in the nucleus, respectively. 
Moreover, we denote by $u_4, \ldots, u_7$ so-called "fictitious concentrations", which \textcolor{black}{describe processes in the nucleus via linear equations yielding a delayed response}. The reaction dynamics of the eight species $u_i$, $i=0,1,\ldots, 7$ are depicted by the diagram in Fig. \ref{JSFig}.
\begin{figure}[htp!]
  %\begin{center}
  %\sidecaption
  \scalebox{1}[1]{
    \begin{tikzpicture}
      \node (a) {$u_1$} node (b) at (1.7,0) {$u_3$} node (c) at (3.4,0) {$u_4$} node (d) at (5.1,0)  {$u_5$} node (e) at (5.1,1.7)  {$u_6$} node (f) at (3.4,1.7)  {$u_7$} node (g) at (1.7,1.7)  {$u_2$} node (h) at (0,1.7)  {$u_0$} ;
      \draw[arrows=->] ([xshift =0.5mm]a.east) -- node [below] {\scalebox{.8}[.8]{$r_{imp2}$}} ([xshift=-0.5mm]b.west);
      \draw[arrows=->] ([xshift =0.5mm]b.east) -- node [below] {\scalebox{.8}[.8]{$r_{delay}$}} ([xshift=-0.5mm]c.west);
      \draw[arrows=->] ([xshift =0.5mm]c.east) -- node [below] {\scalebox{.8}[.8]{$r_{delay}$}} ([xshift=-0.5mm]d.west);
      \draw[arrows=->] ([yshift =0.5mm]d.north) -- node [right] {\scalebox{.8}[.8]{$r_{delay}$}} ([yshift=-0.5mm]e.south);
      \draw[arrows=->] ([xshift =-0.5mm]e.west) -- node [above] {\scalebox{.8}[.8]{$r_{delay}$}} ([xshift=0.5mm]f.east);
      \draw[arrows=->] ([xshift =-0.5mm]f.west) -- node [above] {\scalebox{.8}[.8]{$r_{delay}$}} ([xshift=0.5mm]g.east);
      \draw[arrows=->] ([yshift =-0.5mm]h.south) -- node [left] {\scalebox{.8}[.8]{$r_{act}p_{JAK}$}} ([yshift=0.5mm]a.north);
      \draw[arrows=->] ([xshift =-0.5mm, yshift=0.5mm]g.west) -- node [above] {\scalebox{.8}[.8]{$r_{exp}$}} ([xshift=0.5mm, yshift=0.5mm]h.east);
      \draw[arrows=->] ([xshift =0.5mm, yshift=-0.5mm]h.east) -- node [below] {\scalebox{.8}[.8]{$r_{imp}$}} ([xshift=-0.5mm, yshift=-0.5mm]g.west);
    \end{tikzpicture}\qquad} 
   \caption{Reaction network of the JAK2/STAT5 signalling pathway}\label{JSFig}
  %\end{center}
\end{figure}
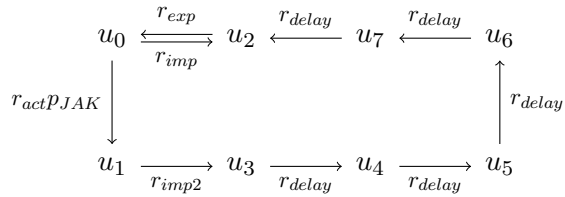

%The ODEs model describing Figure \ref{fig1} writes as
%\begin{equation}\label{ODEmodel}
%	\begin{cases}
%		\nu_{\mathrm{cyt}}u_0'(t) = -\ract p_{JAK}(t) u_0(t) - \rimp u_0(t) + \rexp u_2(t), \\
%		\nu_{\mathrm{cyt}}u_1'(t) = -\rimpp u_1(t) + \ract p_{JAK}(t) u_0(t),\\
%		\nu_{\mathrm{nuc}}u_2'(t) = -\rexp u_2(t) + \rimp u_0(t) + \rdelay u_7(t),\\
%		\nu_{\mathrm{nuc}}u_3'(t) = -\rdelay u_3(t) + \rimpp u_1(t),\\
%		\nu_{\mathrm{nuc}}u_i'(t) = -\rdelay u_i(t) + \rdelay u_{i-1}(t), \quad i=4,5,6,7.
%	\end{cases}
%\end{equation}

The JAK2/STAT5 model considers a smooth, bounded domain $\Omega_0\subset \mathbb R^n$ and distinguishes $\Onuc \subsetneq \Omega_0$ the domain of the cell nucleus and $\Ocyt = \Omega_0\backslash \Onuc$ the cell cytoplasm. With a little abuse of notation, we denote by $\partial\Ocyt = \partial\Omega_0$ the membrane of the cell, while $\partial\Onuc$ is the boundary of the nucleus. 

The following mixed PDE/ODE model was considered  in \cite{FNR13}: The two PDEs
%For the mixed PDE/ODE model, let $\Ocyt\subset \R^n$ be a smooth, bounded domain and $\Onuc$ be a proper subset of $\Ocyt$. We consider the system of two PDEs
\begin{subequations}\label{JAK_system}
\begin{equation}\label{u0} 
  \begin{cases}
    \partial_tu_0(t,x) = D\Delta u_0(t,x), &t>0,\quad x\in\Ocyt,\\
    D\partial_{n_1}u_0(t,y) = -\frac{\ract}{|\partial\Ocyt|}p_{JAK}u_0(t,y), &t>0, \quad y\in\partial\Ocyt,\\
    D\partial_{n_2}u_0(t,y) = -\frac{\rimp}{|\partial\Onuc|}u_0(t,y) + \frac{\rexp}{|\partial\Onuc|}u_2(t), &t>0, \quad y\in\partial\Onuc,
  \end{cases}
\end{equation} 
\begin{equation}\label{u1} 
  \begin{cases}
    \partial_tu_1(t,x) = D\Delta u_1(t,x), &\qquad\quad\ \ t>0,\quad x\in\Ocyt,\\
    D\partial_{n_1}u_1(t,y) = \frac{\ract}{|\partial\Ocyt|}p_{JAK}u_0(t,y), &\qquad\quad\ \ t>0,\quad y\in\partial\Ocyt,\\
    D\partial_{n_2}u_1(t,y) = -\frac{\rimpp}{|\partial\Onuc|}u_1(t,y), &\qquad\quad\ \ t>0, \quad y\in\partial\Onuc,
  \end{cases}
\end{equation} 
\text{and six ODEs}
\begin{equation}\label{ODEs} 
  \begin{cases}
    (u_2)'(t) + \frac{\rexp}{|\Onuc|}u_2(t) = \frac{\rdelay}{|\Onuc|}u_7(t) + \frac{\rimp}{|\Onuc||\partial\Onuc|}\int\limits_{\partial\Onuc}u_0(t,y)dS,\\
    (u_3)'(t) + \frac{\rdelay}{|\Onuc|}u_3(t) = \frac{\rimpp}{|\Onuc||\partial\Onuc|}\int\limits_{\partial\Onuc}u_1(t,y)dS,\\
    (u_i)'(t) + \frac{\rdelay}{|\Onuc|}u_i(t) = \frac{\rdelay}{|\Onuc|}u_{i-1}(t),\quad i=4,5,6,7,
  \end{cases}
\end{equation} 
\end{subequations}
subject to nonnegative initial data $u_0(x,0) = u_0^{in}(x), u_1(x,0) = u_1^{in}(x), \; x\in\Omega$, $u_i(0) = u_i^{in}$ for $i=2,3,\ldots,7$, where $\nu_1$ and $\nu_2$ are outward normals of $\partial\Ocyt$ and $\partial\Onuc$ respectively.
Note that the system \eqref{JAK_system} satisfies the mass conservation law
\begin{equation}\label{Mass} 
  \int\limits_{\Ocyt}(u_0(t,x)+u_1(t,x))dx + |\Onuc|\sum_{i=2}^{7}u_{i}(t) = \int\limits_{\Ocyt}(u_0^{in}(x)+u_1^{in}(x))dx + |\Onuc|\sum_{i=2}^{7}u_{i}^{in}.%,\qquad \forall t>0.
\end{equation}

%\textcolor{black}{We remark that due to the sake of clear presentation, we consider $p_{JAK}$ to be constant. The case $p_{JAK} = p_{JAK}(t)$ is a function of time could be analysed similarly which we do not pursue in this work.}
%\textcolor{red}{I don't think there is an equilibrium state if $p_{JAK} = p_{JAK}(t)$?}

The well-posedness of the mixed PDE/ODE model \eqref{JAK_system} was shown in \cite{FNR13}. Moreover, in the pure ODE case, the authors proved exponential convergence to equilibrium by extensively studying the structure of the reaction matrix. This approach, however, doesn't apply to the PDE/ODE case and the authors were only able to prove Lyapunov stability of the stationary states of \eqref{JAK_system}.
\medskip

The second main result of this paper proves exponential convergence to equilibrium of the mixed PDE/ODE JAK2/STAT5 model \eqref{JAK_system}. 
\begin{theorem}[Exponential convergence to equilibrium of the JAK2/STAT5 model \eqref{JAK_system}]\label{mainresult2}
Let $\Omega_0, \Onuc, \Ocyt \subset \mathbb R^n$, $n\ge 2$ be smooth, bounded domains 
with $\Onuc \subsetneq \Omega_0$, $\Ocyt = \Omega_0\backslash \Onuc$ and $\partial\Ocyt = \partial\Omega_0$.

Then, for any positive initial mass $M>0$, system \eqref{JAK_system} possesses a unique equilibrium $(\ukhong,\ldots, \ubay)$ satisfying the mass conservation law
	\begin{equation}\label{MassInf} 
	  \int\limits_{\Ocyt}(\ukhong(x)+\umot(x))dx + |\Onuc|\sum_{i=2}^{7}u_{i,\infty} = M>0.
	\end{equation}
	Moreover, $\uhai, \ldots, \ubay$ are positive and $\ukhong, \umot\in C(\overline{\Omega}_{\mathrm{cyt}})\cap C^2(\Ocyt)$ satisfy that
	$
		b\leq \ukhong(x), \umot(x) \leq B$ for all $x\in\overline{\Omega}_{\mathrm{cyt}}
	$
	for some constants $0<b\le B\le+\infty$.
	
	Finally, any global weak solution $(u_i)_{i=0,1,\ldots, 7}$ to \eqref{JAK_system} with positive initial mass $M$ (as constructed in \cite{FNR13}) converges exponentially to the equilibrium $(u_{i,\infty})_{i=0,1,\ldots, 7}$  in the sense that
	\begin{multline*}
		\int\limits_{\Ocyt}\left(\frac{|u_0(t,x)-\ukhong(x)|^2}{\ukhong(x)} + \frac{|u_1(t,x)-\umot(x)|^2}{\umot(x)}\right)dx \\+ |\Onuc|\sum_{i=2}^{7}\frac{|u_i(t) - u_{i,\infty}|^2}{u_{i,\infty}} \leq C_1e^{-\lambda_1 t},
	\end{multline*}
	for all $t> 0$, where $\lambda_1$ is as in Lemma \ref{EEDE_JAK} and $C_1, \lambda_1$ are constants, which can be computed explicitly in terms of the domains, parameters and initial mass $M$.
\end{theorem}

\begin{remark}\label{diff_coefficients}
	In system \eqref{JAK_system}, the diffusion coefficients of $u_0$ and $u_1$ are taken the same as  in \cite{FNR13}. However, the proof of Theorem \ref{mainresult2} holds equally for different diffusion coefficients  for $u_0$ and $u_1$, e.g. $u_0$ diffuses with $D_0>0$ and $u_1$ with  $D_1>0$.  %{\color{red}Following ideas of e.g. \cite{DF08,FPT}, it seems quite straightforward to show exponential convergence to equilibrium even if one of these diffusion coefficients is zero.}
\end{remark}
%  ******************************************************* %
%  ******************************************************* %

\section{Preliminaries}%\label{Pre}
The proofs of Theorems \ref{mainresult1} and \ref{mainresult2} use some previous results about 
linear complex balance reaction-diffusion networks proven in \cite{FPT}, which we shall briefly recall in the following.
The interested reader is referred to \cite{FPT} for more details. 

We consider a first order (i.e. linear) reaction network of the form
%\begin{figure}[htp]
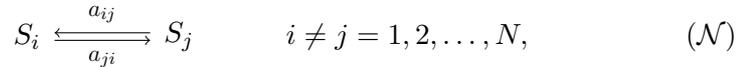
\begin{flushright}
\scalebox{1}[1]{
\begin{tikzpicture}
\node (a) {$S_{i}$} node (b) at (2,0) {$S_{j}$} node (c) at (5,0) {$i\not=j = 1,2,\ldots,N,$} node (d) at (9,0) {$(\mathcal{N})$};
\draw[arrows=->] ([xshift =0.5mm, yshift=-0.5mm]a.east) -- node [below] {\scalebox{.8}[.8]{$a_{ji}$}} ([xshift=-0.5mm,yshift=-0.5mm]b.west);
\draw[arrows=->] ([xshift =-0.5mm,yshift=0.5mm]b.west) -- node [above] {\scalebox{.8}[.8]{$a_{ij}$}} ([xshift=0.5mm,yshift=0.5mm]a.east);
\end{tikzpicture}} 
%\caption{A general first-order chemical reaction network}\label{Reaction}
\end{flushright}
%\end{figure}\\
where $S_i, i = 1,2,\ldots, N$ are different chemical substances (or species) and $a_{ij}, a_{ji} \geq 0$ are constant reaction rates. In particular, $a_{ij}$ denotes the reaction rates from the species $S_j$ to $S_i$. The considered reaction network is contained in a bounded vessel (or reactor) $\Omega\subset\mathbb R^n$, where $\Omega$ is a smooth, bounded domain with outer unit normal $\nu$.
The substances $S_i$ are described by spatio-temporal concentrations $u_i(t,x)$ at position $x\in \Omega$ and time $t\geq 0$.  In addition, each substance $S_i$ is assumed to diffuse with a diffusion coefficient $d_i\geq 0$, $i=1,2,\ldots,N$. 
Finally, using the mass action law as model for the reaction rates leads to the following linear reaction-diffusion system:
\begin{equation}\label{VectorSystem}
	\begin{cases}
		\cc_t = \mathbb D\Delta \cc + A\cc, &\qquad x\in\Omega, \qquad t>0,\\
		\partial_{\nu}\cc = 0, &\qquad x\in\partial\Omega, \qquad t>0,\\
		\cc(x,0) = \cc_0(x)\ge0, &\qquad x\in\Omega,
	\end{cases}
\end{equation}
where $\cc(t,x) = [u_1(t,x), \ldots, u_N(t,x)]^{T}$ denotes the vector of concentrations subject to non-negative initial conditions $\cc_0(x) = [u_{1,0}(x)\ge0, \ldots, u_{N,0}(x)\ge0]^T$, $\mathbb D = \text{diag}(d_1,d_2,\ldots,d_N)$ denotes the diagonal diffusion matrix and the reaction matrix $A = (a_{ij}) \in \mathbb{R}^{N\times N}$ satisfies the following conditions:
\begin{equation}\label{a_jj}
\begin{cases}
	a_{ij} \geq 0, &\qquad \text{for all } i\not=j, \quad i,j =1,2,\ldots,N,\\ 
	a_{jj} = -\sum_{i=1,i\not= j}^{N}a_{ij}, &\qquad \text{for all } j =1,2,\ldots,N.
\end{cases}
\end{equation}
The solution to system \eqref{VectorSystem} satisfies the following mass conservation law
\begin{equation}\label{mass_conservation}
	\sum_{i=1}^{N}\int_{\Omega}u_i(t,x)dx = M:= \sum_{i=1}^{N}\intO u_{i,0}(x)dx >0 \qquad \text{ for all } t>0.
\end{equation}
To study the convergence to equilibrium we consider the following quadratic relative entropy functional
\begin{equation}\label{RelativeEntropy_Recall}
	\mathcal{E}(\cc_1|\cc_2)(t) = \sum_{i=1}^{N}\int_{\Omega}\frac{|u_i|^2}{v_i}dx
\end{equation}
between two solutions $\cc_1=[u_1, \ldots, u_N]^{T}$ and $\cc_2=[v_1, \ldots, v_N]^{T}$ (with respect to possibly different initial data) and its entropy dissipation $\mathcal{D}(\cc_1|\cc_2) = -\frac{d}{dt}\mathcal{E}(\cc_1|\cc_2)$:
\begin{equation}
\mathcal{D}(\cc_1|\cc_2)  
=  2\sum_{i=1}^{N}d_i\int_{\Omega}v_i\left|\nabla\frac{u_i}{v_i}\right|^2dx
+ \sum_{i,j=1; i<j}^{N}\int_{\Omega}(a_{ij}v_j + a_{ji}v_i)\biggl(\frac{u_i}{v_i} - \frac{u_j}{v_j}\biggr)^{\!2}dx.
\label{EntropyDissipation}
\end{equation}

The network $\mathcal N$ is called {\it weakly reversible} if for any reaction $S_i \xrightarrow{a_{ji}} S_j$ with $a_{ji}>0$, there exist $S_{k_1}, \ldots, S_{k_r}$ such that $S_j \equiv S_{k_0} \xrightarrow{a_{k_1k_0}} S_{k_1} \xrightarrow{a_{k_2k_1}} S_{k_2} \rightarrow \ldots \rightarrow S_{k_r} \xrightarrow{a_{k_{r+1}k_r}} S_i \equiv S_{k_{r+1}}$ with $a_{k_{i+1}k_{i}} >0$ for all $i=0, 1, \ldots, r$. Intuitively, a network $\mathcal N$ is weakly reversible if for any reaction from $S_i$ to $S_j$ we can find a returning chain of reactions, which starts from $S_j$ and finishes at $S_i$. 

%\textcolor{red}{
If $\mathcal N$ is weakly reversible, then the associated reaction graph can be composed of multiple disjoint strongly connected components, see \cite{FPT}. However, these components are entirely independent since every node of a first order reaction network represents only one species. Therefore, all such disjoint components can be treated separately from one another and w.l.o.g. we {\it shall consider in the following  only weakly reversible networks consisting of one strongly connect component}. Finally, all weakly reversible first order reaction networks satisfy the complex balance condition, see \cite{FPT}.
%}
The following theorem is one of the main results in \cite{FPT}.
\begin{theorem}[Exponential convergence to equilibrium for first order reaction networks, \cite{FPT}]\
Assume that the reaction network $\mathcal N$ is weakly reversible {and consists w.l.o.g. of only one strongly connect component}. Moreover, assume that there is at least one positive diffusion coefficient, that is, there exists $i_0$ such that $d_{i_0}>0$. 
Let $\Omega \subset \mathbb R^n$, $n\ge1$ be a bounded domain with sufficiently smooth boundary $\partial\Omega$ (e.g. $\partial\Omega \in C^{2+\epsilon}$, $\epsilon>0$).

Then, for any given positive initial mass $M>0$, system \eqref{VectorSystem} possesses a unique positive complex balance equilibrium $\cc_{\infty} = (u_{1,\infty}, \ldots, u_{n,\infty})$ satisfying the mass conservation law \eqref{mass_conservation}. Moreover, each solution to \eqref{VectorSystem} with positive initial mass $M$ converges exponentially to this equilibrium with computable rates, i.e.
	\begin{equation*}
		\sum_{i=1}^{N}\|u_i(t) - u_{i,\infty}\|_{L^2(\Omega)}^2 \leq Ce^{-\lambda t} \qquad \text{ for all } t>0,
	\end{equation*}
	where $C, \lambda$ are constants, which can be computed explicitly.
\end{theorem}

The following elementary inequality, which was proved in \cite{FPT} (in a variant), will be useful in the following sections.
\begin{lemma}[A finite dimensional inequality, see \cite{FPT}]\label{finite_dimension} \
Assume that the network $\mathcal N$ is weakly reversible {and consists of one strongly connected component}.
  
Then,  for all $\cc = (c_1, \ldots, c_N)$ satisfying
$\sum_{i=1}^{N}\alpha_ic_i = 0$
with $\alpha_1, \ldots, \alpha_N$ being positive constants, there holds
\begin{equation*}
	\sum_{i,j=1; i<j}^{N}(a_{ij} + a_{ji})(c_i - c_j)^2 \geq \eta\sum_{i=1}^{N}c_i^2
\end{equation*}
for an {\normalfont{explicit}} constant $\eta >0$ depending only on $\alpha_i$, $a_{ij}$ and $N$.
\end{lemma}
\begin{proof}
	First, thanks to the weak reversibility of $\mathcal N$, and since $\mathcal N$ contains only one strongly connected component, we observe for every $S_i$ and $S_j$ in $\mathcal N$ that there exist a chain of nontrivial reactions starting from $S_i$ and finishing at $S_j$ and vice versa. Hence, by using iteratively the triangle inequality along such chains of reactions in cases where $a_{ij} = a_{ji} = 0$, we can estimate (see \cite[Lemma 2.4]{FPT} for the details) 
	\begin{equation}\label{e1}
		\sum_{i,j=1; i<j}^{N}(a_{ij} + a_{ji})(c_i - c_j)^2 \geq \zeta\sum_{i,j=1;i<j}^{N}(c_i - c_j)^2
	\end{equation}
for some explicit constant $\zeta(N)>0$ only depending on ${N}$.
Now, for any $1\leq i_0\leq N$
\begin{align}
&\sum_{i,j=1;i<j}^{N}(c_i - c_j)^2 \geq \sum_{j=1; j\not=i_0}^{N}(c_{i_0} - c_j)^2 \geq \sum_{j=1; j\not=i_0}^{N}\frac{1}{\alpha_j^2}(\alpha_jc_{i_0} - \alpha_jc_j)^2\nonumber\\
			&\qquad\geq \frac{\bigl(c_{i_0}\sum_{j=1;j\not=i_0}^{N}\alpha_j - \sum_{j=1;j\not=i_0}^{N}\alpha_jc_j\bigr)^2}{(N-1)\max_{i=1,\ldots,N}\{\alpha_i^2\}}
			= \frac{\bigl(\sum_{j=1}^{N}\alpha_j\bigr)^2}{(N-1)\max_{i=1,\ldots,N}\{\alpha_i^2\}}\,c_{i_0}^2.\label{e2}
\end{align}
	Since $1\leq i_0\leq N$ is arbitrary, by combining \eqref{e1} and \eqref{e2} we get the desired result.
%	\begin{equation*}
%		\sum_{i,j=1;i<j}^{N}(a_{ij}+a_{ji})(c_i - c_j)^2 \geq \frac{\zeta\left(\sum_{i=1}^{N}\alpha_i\right)^2}{N(N-1)\max_{i=1,\ldots,N}\{\alpha_i^2\}}\sum_{i=1}^{N}c_i^2
%	\end{equation*}
%	which is the desired result.
\end{proof}

\section*{Proof of Theorem \ref{mainresult1}}
In this section, we denote by $\cc = (L, P, \ell, p)$ the vector of concentrations of system \eqref{Lgl_system}. Following the previous section, we introduce the relative entropy between two solution trajectories $\cc_1= (L_1, P_1, \ell_1, p_1)$ and $\cc_2 = (L_2, P_2, \ell_2, p_2)$ as follow
\begin{equation}\label{Lgl_entropy}
	\E(\cc_1|\cc_2) = \intO\left(\frac{L_1^2}{L_2} + \frac{P_1^2}{P_2}\right)dx + \intG \frac{\ell_1^2}{\ell_2}dS + \intGG \frac{p_1^2}{p_2}dS.
\end{equation}
Corresponding to \eqref{EntropyDissipation}, we can compute the entropy dissipation functional corresponding to \eqref{Lgl_entropy}, i.e $\D(\cc_1|\cc_2) = -\frac{d}{dt}\E(\cc_1|\cc_2)$  as
\begin{align}
		&\D(\cc_1|\cc_2) %&= -\frac{d}{dt}\E(\cc_1|\cc_2)\\
			 = 2d_L\intO{L_2\left|\nabla\frac{L_1}{L_2}\right|^2dx} + 2d_P\intO{P_2\left|\nabla\frac{P_1}{P_2}\right|^2dx} + 2d_{\ell}\intG{\ell_2\left|\nabla_{\Gamma}\frac{\ell_1}{\ell_2}\right|^2dS}\nonumber\\
			 &+ 2d_p\intGG{p_2\left|\nabla\frac{p_1}{p_2}\right|^2dS}
			  + \intO{(\alpha P_2 + \beta L_2) \left|\frac{L_1}{L_2}-\frac{P_1}{P_2}\right|^2dx} \nonumber\\
			 &+ \intG{\!(\lambda L_2+\gamma \ell_2)\left|\frac{L_1}{L_2} - \frac{\ell_1}{\ell_2}\right|^2\!dS}+ \intGG{\!\xi p_2\left|\frac{P_1}{P_2} - \frac{p_1}{p_2}\right|^{2}\!dS} + \intGG{\!\sigma\ell_2\left|\frac{\ell_1}{\ell_2} - \frac{p_1}{p_2}\right|^2\!dS}.\label{Lgl_dissipation}
	\end{align}

The following Lemma \ref{Lgl_equilibrium} proves the existence of a unique positive equilibrium provided positive initial mass 
$M>0$. The main difficulties are the complex balance structure of the equilibrium (with the associated mass conservation law preventing standard coercivity arguments) and the mixed boundary conditions impeding classical solutions and thus, the direct use of classical maximum principles.  
   
\begin{lemma}[Existence of a unique positive equilibrium]\label{Lgl_equilibrium}\hfill\\
	Let $\Omega$ be a bounded domain of $\mathbb R^n$, $n=2,3$ with smooth boundary $\Gamma = \partial\Omega$. 
	
Then, for any positive initial mass $M>0$, the system \eqref{Lgl_system} possesses a unique positive equilibrium $\cc_{\infty} = (\Li, \Pi, \li, \pi)$ satisfying the mass conservation
	\begin{equation}\label{Lgl_massinf}
		\int_{\Omega}(\Li(x) + \Pi(x))dx + \int_{\Gamma}\li(x)dS + \int_{\Gamma_2}\pi(x)dS = M.
	\end{equation}
	Moreover, $\Li \in C(\overline{\Omega})\cap H^2(\Omega)$, $\Pi\in L^{\infty}(\Omega)\cap H^{3/2}(\Omega)$, $\li\in C({\Gamma})\cap H^2(\Gamma)$ and $\pi\in C(\overline{\Gamma_2})\cap H^2(\Gamma_2)$, and satisfy
	for some constants $0<a\leq A <+\infty$
	\begin{equation*}
			0 < a \leq \Li(x)\leq A \quad \text{ for all } x\in\overline{\Omega},
	\end{equation*}
	\begin{equation*}
			0 < a \leq \Pi(x)\leq A \quad \text{ for almost all } x\in\Omega,
	\end{equation*}
	\begin{equation*}
		0 < a \leq \li(x) \leq A \quad \text{ for all } x\in{\Gamma},
	\end{equation*}
	%and
	\begin{equation*}
		0 < a \leq \pi(x) \leq A \quad \text{ for all } x\in\overline{\Gamma_2}.
	\end{equation*}
	
\end{lemma}
\begin{proof}
	We will first prove that the equilibrium system has non-negative solutions and then show that equilibria are indeed bounded and strictly positive. Finally, the uniqueness follows from the vanishing of entropy-dissipation functional \eqref{Lgl_dissipation}.
	
In order to prove the existence of nonnegative equilibria via a fixed point argument, we consider the following auxiliary system
	\begin{equation}\label{iteration}
		\begin{cases}
			-d_L \Delta L + \beta L = \alpha P_0, \quad -d_P\Delta P + \alpha P = \beta L_0,&x\in\Omega,\\
%			-d_P\Delta P + \alpha P = \beta L_0, &x\in\Omega,\\
			\ \ \ \,d_L\partial_{\nu}L + \lambda L= \gamma \ell_0, \quad\ \ \ d_P\partial_{\nu}P =  \chi_{\Gamma_2} \xi p_0, &x\in\Gamma,\\
%			\qquad d_P\partial_{\nu}P =  \chi_{\Gamma_2} \xi p_0, &x\in\Gamma,\\
			-d_{\ell}\Delta_{\Gamma}\ell + (\lambda + \sigma \chi_{\Gamma_2})\ell = \lambda L_0|_{\Gamma}, &x\in\Gamma,\\
			-d_p\Delta_{\Gamma_2}p + \xi p = \sigma \ell_0, &x\in\Gamma_2,\\
			\qquad d_p\partial_{\nu_{\Gamma_2}}p = 0, &x\in\partial\Gamma_2
		\end{cases}
	\end{equation}
	where $(L_0, P_0, \ell_0, p_0) \in \mathcal Y$ are given in the space
	\begin{equation*}
		\mathcal Y = \{(U,V,u,v)\in H^1(\Omega)\times L^2(\Omega)\times L^2(\Gamma)\times L^2(\Gamma_2): U, V, u, v\geq 0\}.
	\end{equation*}
	By standard linear elliptic equation theory, there exists a unique weak solution $(L, P, \ell,p)\in H^1(\Omega)\times H^1(\Omega)\times H^1(\Gamma)\times H^1(\Gamma_2)$ for \eqref{iteration}. Thanks to the nonnegativity of $(L_0, P_0, \ell_0, p_0) \in \mathcal Y$ and the weak maximum principle (cf. e.g. \cite{GT}), this solution is also nonnegative: Indeed, by testing, for instance, the equation for $P$ by the negative part $P_{-}=-\min\{P,0\}$, 
we calculate with $P P_{-} = -(P_{-})^2$
\begin{equation}\label{weakMax}
%- \int_{\Gamma}  \chi_{\Gamma_2} \xi p_0 P_{-}\,dS + d_P \int_{\Omega} |\nabla P|^2 (P_{-})' \,dx= 
%+ \alpha \int_{\Omega} (P_{-})^2 + \beta  \int_{\Omega} L_0 P_{-}\,dx,
- \int_{\Gamma}  \chi_{\Gamma_2} \xi p_0 P_{-}\,dS - d_P \int_{\Omega} {\chi_{\{P\leq 0\}}}|\nabla P|^2  \,dx= 
+ \alpha \int_{\Omega} (P_{-})^2 + \beta  \int_{\Omega} L_0 P_{-}\,dx,
\end{equation}
and observe that the left hand side is nonpositive while the right hand side is nonnegative provided that $p_0$ and $L_0$ are nonnegative. Thus, both sides have to equal zero and, as a consequence, $\int_{\Omega} (P_{-})^2=0$, which implies the 
nonnegativity of $P$. 

Moreover, the smoothness of the boundary $\partial\Omega$ allows to deduce higher regularity for $L$, namely $L\in H^2(\Omega)$ thanks to $P_0\in L^2(\Omega)$ and $\ell_0\in L^2(\Gamma)$. In particular, the following {\it a prior} estimate holds
	\begin{multline*}
		\|L\|_{H^2(\Omega)} + \|P\|_{H^1(\Omega)} + \|\ell\|_{H^1(\Gamma)} + \|p\|_{H^1(\Gamma_2)}\\ \leq C\left(\|L_0\|_{H^1(\Omega)} + \|P_0\|_{L^2(\Omega)} + \|\ell_0\|_{L^2(\Gamma)} + \|p_0\|_{L^2(\Gamma_2)}\right).
	\end{multline*}
	By defining $\mathfrak T: \mathcal Y \rightarrow \mathcal Y$ by $\mathfrak T(L_0,P_0,\ell_0,p_0) = (L, P, \ell, p)$, we obtain from the previous {\it a priori} estimate that $\mathfrak T$ is a compact operator. Hence, it follows from the Schauder fixed point theorem that there exist fixed points $(\Li, \Pi, \li, \pi)$ of $\mathfrak T$, and these fixed points are thus nonnegative solutions to the equilibrium system. Note that uniqueness of the fixed points $(\Li, \Pi, \li, \pi)$ can not be supposed as such, since we expect equilibria to exist for any given mass $M$. 	
\medskip

	Next, again by maximal regularity for linear elliptic equations, we obtain $\Li\in H^2(\Omega)$, $\li \in H^2(\Gamma)$ and $\pi\in H^2(\Gamma_2)$, which implies $\Li\in C(\overline{\Omega})$, $\li\in C({\Gamma})$ and $\pi \in C(\overline{\Gamma_2})$ thanks to Sobolev embeddings and $\Omega \subset \mathbb R^n$ with $n=2,3$.
The continuity of $\Li, \li$ and $\pi$ and the compactness of $\overline{\Omega}$, $\Gamma$ and $\overline{\Gamma_2}$ imply also the upper bounds 
$\Li, \li, \pi \leq A <+\infty$ for a constant $A$. 

For $\Pi$ satisfying the mixed Neumann boundary condition $d_{P}\partial_{\nu}\Pi = \chi_{\Gamma_2}\pi$ with $\chi_{\Gamma_2}\in L^{\infty}(\Gamma)$ being discontinuous, maximal elliptic regularity only yields $\Pi\in H^{3/2}(\Omega)$, which is insufficient to conclude boundedness in three space dimensions, see e.g. \cite{Evans}.
However, we are able to construct supersolutions $\hat{P}$ as the solutions of  
\begin{equation*}
\begin{cases}
-d_P\Delta \hat{P} + \alpha \hat{P} = \beta \|\Li\|_{\infty}, &x\in\Omega,\\
d_P\partial_{\nu}\hat{P} = \xi \|\pi\|_{\infty}, &x\in\partial\Gamma.
\end{cases}
\end{equation*}
For the supersolutions $\hat{P}$, standard elliptic theory implies 
the maximal regularity $\hat{P}\in H^2(\Omega)$ und thus continuity and
boundedness. Moreover, the same weak maximum principle argument as in \eqref{weakMax} yields $\hat{P}\ge \Pi$ and thus the upper bound $\Pi  \leq A <+\infty$ for a constant $A$. 

We will show now that $\li \equiv 0$ implies $\Li = \Pi = \pi = 0$. Indeed, with $\li \equiv 0$ it follows readily that $\pi \equiv 0$. By multiplying the equation for $\Li$ by $\beta \Li$, the equation for $\Pi$ by $\alpha \Pi$, and by summing the two equations, we calculate with $\pi \equiv 0$
	\begin{equation*}
		d_{L}\|\nabla \Li\|_{L^2(\Omega)}^2 + \lambda\|\Li\|_{L^2(\Gamma)}^2 + d_P\|\nabla \Pi\|_{L^2(\Omega)}^2 + \int_{\Omega}(\beta \Li- \alpha \Pi)^2\,dx= 0,
	\end{equation*}
which implies $\Li \equiv 0$ and eventually $\Pi \equiv 0$. 
Therefore, whenever a positive mass $M>0$ is considered, 
the corresponding equilibrium state has to satisfy $\li \not\equiv 0$, and consequently $\pi\not\equiv 0$ and $\Pi \not\equiv 0$. 

From the continuity of $\li$, we obtain that $\pi$ is the unique classical solution to 
\begin{equation*}
		\begin{cases}
			-d_{p}\Delta_{\Gamma_2}\pi + \xi\pi  = \sigma \li, &x\in\Gamma_2,\\
			d_p\partial_{\nu_{\Gamma_2}}\pi = 0, &x\in\partial\Gamma_2
		\end{cases}
\end{equation*}
Since $\li$ is nonnegative and not identically zero, we can apply the classical maximum principle to conclude $\pi(x) \geq a>0$ for $x\in\overline{\Gamma_2}$ and a constant $a>0$.

Next, by considering the auxiliary equation 
	\begin{equation*}
		-d_{\ell}\Delta_{\Gamma}\ell^* + (\gamma + \sigma)\ell^* = \lambda \Li|_{\Gamma}, \quad x\in\Gamma,
	\end{equation*}
and by recalling the continuity and nonnegativity of $\Li\not\equiv 0$, the strong maximum applied to the unique classical solution $\ell^*$ implies that $\ell^*(x) \geq a>0$ for all $x\in\Gamma$ and a constant $a>0$. Moreover, by a weak maximum principle argument analog to \eqref{weakMax}, we have that $\ell^*$ is a subsolution to $\li$, i.e. $\li(x) \geq \ell^*(x) \geq a>0$ for all $x\in\Gamma$.

Moreover, we consider the unique classical solutions $L^*$ of the auxiliary system
\begin{equation*}
%\begin{cases}
-d_L\Delta L^* + \beta L^* = 0, \quad x\in\Omega,\qquad
d_L\partial_{\nu} L^* + \lambda L^* = \gamma \li, \quad x\in\Gamma,
%\end{cases}
\end{equation*}
for which the classical maximum principle and the lower bound $\li(x) \geq a>0$ implies  
$L^*(x) \geq a>0$ for $x\in{\Omega}$ and a constant $a>0$. Furthermore, by the weak maximum principle, 
$L^*$ is a subsolution to $\Li$, i.e. $\Li(x) \geq L^* \geq a>0$ for all $x\in{\Omega}$.

%From the continuity of $\Pi$ and $\li$, we can obtain that $\Li$ and $\pi$ are actually classical solutions to 
%	\begin{equation*}
%		\begin{cases}
%			-d_L\Delta \Li + \beta \Li = \alpha \Pi, &x\in\Omega,\\
%			d_L\partial_{\nu}\Li + \lambda \Li = \gamma \li, &x\in\Gamma
%		\end{cases}
%	\end{equation*}
%	and
%	\begin{equation*}
%		\begin{cases}
%			-d_{p}\Delta_{\Gamma_2}\pi + \xi\pi  = \sigma \li, &x\in\Gamma_2,\\
%			d_p\partial_{\nu_{\Gamma_2}}\pi = 0, &x\in\partial\Gamma_2
%		\end{cases}
%	\end{equation*}
%	respectively. Then with $\Pi, \li$ are nonnegative and not identically zero, we can apply the classical maximum principle to get that $\Li(x) >0$ for $x\in\overline{\Omega}$ and $\pi(x) >0$ for $x\in\overline{\Gamma_2}$. Therefore $\Li(x) \geq a_1 >0$ for $x\in\overline{\Omega}$ and $\pi(x) \geq a_1 >0$ for $x\in\overline{\Gamma_2}$, for some $a_1>0$. 
	
Finally, by considering the unique classical solution $P^*$ of the auxiliary system
	\begin{equation*}
%		\begin{cases}
			-d_P\Delta P^* + \alpha P^* = \beta \Li, \quad x\in\Omega,\qquad
			d_P\partial_{\nu}P^* = 0, \quad x\in\partial\Gamma,
%		\end{cases}
	\end{equation*}
the weak maximum principle shows $P^*$ to be subsolutions, which is bounded below by a positive constants due to the 
strong maximum principle applied to $P^*$, i.e. $\Pi(x) \geq P^*(x)\geq a >0$ for $x\in{\Omega}$ and a constant $a>0$.
This finishes the proof of the lower and upper bounds. 
%\smallskip

%	Since $\Li\in C(\overline{\Omega})$ then there exists a unique classical solution $P^*$ and moreover $P^*(x) > 0$ for $x\in\overline{\Omega}$ thanks to $\Li >0$. By a weak comparison principle, it's standard that $\Pi(x) \geq P^*(x)$ for $x\in\overline{\Omega}$ which implies $\Pi(x) \geq a_2 >0$ for all $x\in\overline{\Omega}$, for some $a_2>0$. Similarly, by considering the auxiliary equation
%	\begin{equation*}
%		-d_{\ell}\Delta_{\Gamma}\ell^* + (\gamma + \sigma)\ell^* = \lambda \Li|_{\Gamma}, \quad x\in\Gamma
%	\end{equation*}
%	we obtain that $\li(x) \geq \ell^*(x) > a_3$ for all $x\in\Gamma$, for some $a_3>0$.
	
	To prove the uniqueness of the equilibrium for a given positive mass $M>0$, we suppose two different equilibria $\cc^{(1)}_{\infty} = (\Li^{(1)}, \Pi^{(1)}, \li^{(1)}, \pi^{(1)})$, $\cc_{\infty} = (\Li, \Pi, \li, \pi)$ as constructed above. Then, obviously the entropy-dissipation of the relative entropy between $\cc^{(1)}$ and $\cc_{\infty}$ vanishes, i.e.
$\D(\cc_{\infty}^{(1)}, \cc_{\infty}) = 0$.
	Thanks to \eqref{Lgl_dissipation}, this implies
	\begin{equation*}
		\frac{\Li^{(1)}}{\Li} \equiv \frac{\Pi^{(1)}}{\Pi} \equiv \frac{\li^{(1)}}{\li} \equiv \frac{\pi^{(1)}}{\pi} \equiv k
	\end{equation*}
	for some constant $k\in \mathbb R\backslash\{0\}$. Hence, the conservation law \eqref{Lgl_massinf} implies $\cc_{\infty}^{(1)} \equiv \cc_{\infty}$ provided a fixed positive mass $M>0$. 	
\end{proof}

\begin{remark}
Note that the existence of a nonnegative equilibrium is proved independently of the space dimension. 
The positive lower and upper bounds, however, are based on classical maximum principles arguments. 
Due to the discontinuity of the characteristic function $\chi_{\Gamma_2}$, we do not get classical solutions but only weak solutions with higher regularity (e.g. $\Li \in H^2(\Omega)$, $\Pi\in H^{3/2}(\Omega)$), which restricts our proof to  dimensions $n\le3$. The case of higher spatial dimensions remains open.
\end{remark}

The proof of Theorem \ref{mainresult1} is based on the following crucial Lemma \ref{EEDE_Lgl}, which establishes a so called \emph{entropy entropy-dissipation estimate} and constitutes 
the key idea of the entropy method, which aims to quantify the entropy dissipation in terms of the 
relative entropy towards the equilibrium via a functional inequality independent of the flow of a PDE model, see e.g. \cite{DF06, DF08}.

In order to prove the entropy entropy-dissipation estimate in Lemma \ref{EEDE_Lgl} and in particular as a consequence of 
having to prove an entropy method for the space inhomogeneous equilibria  $\cc_{\infty} = (\Li, \Pi, \li, \pi)$, it will be highly convenient to introduce the following abbreviations, weighted quantities and inequalities: 
\begin{itemize}
	\item Norms: $\|\cdot\|_{\Omega}, \|\cdot\|_{\Gamma}$, $\|\cdot\|_{\Gamma_2}$ are the norms in $L^2(\Omega)$, $L^2(\Gamma)$, $L^2(\Gamma_2)$, respectively;
	\item New variables (weighted deviations around equilibrium values):
	$$
		U = \frac{L - \Li}{\Li},\quad V = \frac{P - \Pi}{\Pi}, \quad u = \frac{\ell - \li}{\li}, \quad v = \frac{p - \pi}{\pi};
	$$
	\item New measures:
	$\qquad\qquad \quad
		d\Li = \Li dx, \quad \quad d\Pi = \Pi dx, 
	$
	$$
		dS_{\Li} = \Li\!\!\mid_{\Gamma} dS, \qquad dS_{\Pi} = \Pi\!\!\mid_{\Gamma} dS, 
	\qquad	d\li = \li dS, \qquad d\pi = \pi dS.
	$$
	\item Weighted averages:\quad
	$
		\oU= \frac{1}{\intO{d\Li}}\intO{Ud\Li}, \quad \oV = \frac{1}{\intO{d\Pi}}\intO{Vd\Pi},
	$
	$$
		\qquad\qquad\ou = \frac{1}{\intG{d\li}}\intG{\ell d\li}, \quad \ov = \frac{1}{\intGG{d\pi}}\intGG{pd\pi}.
	$$
%\end{itemize}
%We will also use the following useful inequalities
%\begin{itemize}
	\item Weighted Poincar\'e Inequalities: \textcolor{black}{The following weighted inequalities hold thanks to the upper and lower bounds of $\Li, \Pi, \li$ and $\pi$ in Lemma \ref{Lgl_equilibrium}}
	\begin{equation}\label{t1}
		\intO{|\nabla U|^2d\Li} \geq P_L\!\intO{|U - \oU|^2d\Li},\ \ \ \intO{|\nabla V|^2d\Pi} \geq P_P\!\intO{|V - \oV|^2d\Pi},
	\end{equation}
	\begin{equation}\label{t2}
		\intG{|\nabla_{\Gamma}u|^2d\li}\geq P_{\ell}\!\intG{|u - \ou|^2d\li}, \quad
		\intGG{|\nabla_{\Gamma_2} v|^2d\pi} \geq P_p\!\intGG{|v - \ov|^2d\pi}.	
	\end{equation}
	\item Weighted Trace Inequalities: \textcolor{black}{Thanks to the lower and upper bounds of $\Li\in C(\overline{\Omega})$ and the usual Trace inequality, we have}
	\begin{eqnarray}\label{t4}
		\intO{|\nabla U|^2d\Li} \geq T_L\!\!\intG{\!|U\!\!\mid_{\Gamma}\! - \oU|^2dS_{\Li}} %, \intO{|\nabla V|^2d\Pi} \geq T_P\!\!\intG{\!|V\!\!\mid_{\Gamma}\! - \oV|^2dS_{\Pi}}
	\end{eqnarray}
%	in which the first one holds due to the upper and lower bound of $\Li$ in Lemma \ref{Lgl_equilibrium} while the validity of the second one can be seen in e.g. \cite{Auc}.
\end{itemize}

With respect to the new notations, note that the relative entropy \eqref{Lgl_entropy}, in particular the relative entropy w.r.t. the equilibrium  $\cc_{\infty}$, i.e. $\E(\cc |\cc_{\infty})$ can be rewritten as
\begin{equation}\label{entropyshift}
\E(\cc |\cc_{\infty}) = \E(\cc - \cc_{\infty}|\cc_{\infty}) +M, 
\qquad \text{where}\quad \E(\cc_{\infty}|\cc_{\infty})=M,
\end{equation}
and that 
%\begin{equation*}%\label{EntropyR}
$\E(\cc - \cc_{\infty}|\cc_{\infty}) = \intO{U^2d\Li} + \intO{V^2d\Pi} + \intG{u^2d\li} + \intGG{v^2d\pi}$.
%\end{equation*}
Moreover, the entropy dissipation law \eqref{Lgl_dissipation} 
rewrite as 
\begin{align}
&\D(\cc - \cc_{\infty}|\cc_{\infty}) = -\frac{d}{dt}\E(\cc - \cc_{\infty}|\cc_{\infty}) \nonumber\\
%\D(\cc - \cc_{\infty}|\cc_{\infty}) 
&\;= 2d_L\intO{|\nabla U|^2d\Li} + 2d_P\intO{|\nabla V|^2d\Pi}+ 2d_{\ell}\intG{|\nabla_{\Gamma} u|^2d\li} + 2d_p\intGG{|\nabla_{\Gamma_2} v|^2d\pi}\nonumber\\
			&\;\;\;\; \;+ \alpha\intO{|U - V|^2d\Pi}  + \beta\intO{|U-V|^2d\Li} + \lambda\intG{|U\!\!\mid_{\Gamma} - u|^2dS_{\Li}}\nonumber\\
			&\;\;\;\;\; + \gamma\intG{|U\!\!\mid_{\Gamma} - u|^2d\li} + \xi\intGG{|V\!\!\mid_{\Gamma} - v|^2d\pi} + \sigma\intGG{|u - v|^2d\li},\label{EntropyDissipationR}
\end{align}
and from \eqref{entropyshift} it follows readily that 
$\D(\cc - \cc_{\infty}|\cc_{\infty})=\D(\cc|\cc_{\infty})$.

\begin{proof}[Proof of Theorem \ref{mainresult1}]\hfill\\
The proof of Theorem \ref{mainresult1} is a direct consequence of the  key functional inequality \eqref{EEDE} in the following  Lemma \ref{EEDE_Lgl}  
and a classical Gronwall argument applied to the entropy dissipation law \eqref{EntropyDissipationR}, which is weakly satisfied (in the sense of being integrated in time)
by the weak global solutions to system \eqref{Lgl_system} constructed in \cite{FRT16}.
\end{proof}

\begin{lemma}[Entropy entropy-dissipation estimate for system \eqref{Lgl_system}]\label{EEDE_Lgl}\hfill\\
	Fix a positive initial mass $M>0$. Then, for any non-negative measurable functions $\cc = (L, P, \ell, p)$ satisfying the mass conservation
	\begin{equation*}
		\int_{\Omega}(L(x) + P(x))dx + \intG \ell(x)dS + \intGG p(x)dS = M,
	\end{equation*}
	the entropy entropy-dissipation estimate
	\begin{equation}\label{EEDE}
		\D(\cc - \cc_{\infty}|\cc_{\infty}) \geq \lambda_0\, \E(\cc - \cc_{\infty}|\cc_{\infty})
	\end{equation}
	holds, where $\cc_{\infty}$ is as in Lemma \ref{Lgl_equilibrium} and the constant $\lambda_0>0$ can be estimated explicitly.
\end{lemma}
\begin{proof}
	Note that $\D(\cc - \cc_{\infty}|\cc_{\infty})=0$ for all constant states satisfying $U = V = u = v$ while $\E(\cc - \cc_{\infty}|\cc_{\infty}) = 0$ if and only if $U = V = u = v = 0$. Hence, the constraint provided by mass conservation law, i.e. 
\begin{equation*}
\int_{\Omega}U(x)d\Li + \int_{\Omega}V(x)d\Pi + \intG u(x)d\li + \intGG v(x)d\pi = 0,
\end{equation*}
plays a crucial role in inequality \eqref{EEDE}, which can not hold otherwise. 

The proof of this lemma is therefore divided into two steps, where the mass conservation law enters the proof in the first step. At first, we remark that the relative entropy enjoys to following additivity property w.r.t. $\overline{\cc} = (\overline{U}, \overline{V}, \overline{u}, \overline{v})$
\begin{equation*}
\E(\cc - \cc_{\infty}|\cc_{\infty}) = \E(\cc - \overline{\cc}|\cc_{\infty}) + \E(\overline{\cc} - \cc_{\infty}|\cc_{\infty})
\end{equation*}
and that the second term on the right hand side is controlled in terms of the entropy dissipation in Step 1, while the first term is controlled in Step 2:
		
	\medskip
	\noindent{\bf Step 1.} First, we prove that there exists an explicit constant $K_0>0$ such that
	\begin{equation}\label{Eq1}
		\D(\overline{\cc} - \cc_{\infty}|\cc_{\infty}) \geq K_0\,\E(\overline{\cc} - \cc_{\infty}|\cc_{\infty})
	\end{equation}
	Indeed, \eqref{Eq1} writes explicitly as
	\begin{multline*}
		%\begin{aligned}
			\alpha|\oU-\oV|^2\intO{d\Pi} + \beta|\oU - \oV|^2\intO{d\Li} + \lambda|\oU - \ou|^2\intG{dS_{\Li}}\\
			+\gamma|\oU - \ou|^2\intG{d\li} + \xi|\oV - \ov|^2\intGG{d\pi} + \sigma|\ou - \ov|^2\intGG{d\li}\\
			\geq K_0\Bigl(\oU^2\intO{d\Li} + \oV^2\intO{d\Pi} + \ou^2\intG{d\li} + \ov^2\intGG{d\pi}\Bigr)
		%\end{aligned}
	\end{multline*}
	under the mass constrain
	%\begin{equation*}
		$\oU\intO{d\Li} + \oV\intO{d\Pi} + \ou\intG{d\li} + \ov\intGG{d\pi} = 0$ {and where $\alpha, \beta, \lambda, \gamma, \sigma$ and $\xi$ denote the positive reaction rate constants
of the network Fig. \ref{LGLmodel}}.
	%\end{equation*}

However, the above inequality is a direct consequence of Lemma \ref{finite_dimension} applied to the vector of averaged concentrations $(\oU, \oV, \ou, \ov)$
after noting that the network of reactions between $L, P, \ell$ and $p$ (see Figure \ref{LGLmodel}) is weakly reversible with one strongly connected component. Thus, the constants $K_0$ 
can be taken as the corresponding constant $\eta>0$ of Lemma \ref{finite_dimension}.
\medskip

\noindent{\bf Step 2.} We introduce the following deviations from the averaged values by $\delta_U = U - \oU, \delta_V = V - \oV, \delta_u = u - \ou$ and $\delta_v = v - \ov$. Note that
	$
		\intO{\delta_Ud\Li} = \intO{\delta_Vd\Pi} = \intG{\delta_ud\li} = \intGG{\delta_vd\pi} = 0.
	$
Moreover, we can rewrite
	\begin{equation}\label{Eq2}
		\begin{aligned}
			\E(\cc-\cc_{\infty}|\cc_{\infty}) = \intO{\delta_U^2d\Li} + \intO{\delta_v^2d\Pi} + \intG{\delta_u^2d\li} + \intGG{\delta_v^2d\pi} + \E(\overline{\cc} - \cc_{\infty}|\cc_{\infty}).
		\end{aligned}
	\end{equation}
	By using the weighted Poincar\'{e} and Trace inequalities  \eqref{t1}--\eqref{t4}, we estimate $\D$ as 
	\begin{equation}\label{Eq3}
		\begin{aligned}
			\D(\cc-\cc_{\infty}|\cc_{\infty})&\geq d_LP_L\intO{\delta_U^2d\Li} + d_PP_P\intO{\delta_V^2d\Pi}+ d_LT_L\intG{\delta_U^2dS_{\Li}} \\
				&\;\;\; + 2d_{\ell}P_{\ell}\intG{\delta_u^2d\li} + 2d_pP_p\intGG{\delta_v^2d\pi}\\
				&\;\;\; + \biggl[\alpha\intO{|U - V|^2d\Pi}  + \beta\intO{|U-V|^2d\Li} + \lambda\intG{|U\!\!\mid_{\Gamma} - u|^2dS_{\Li}}\\
				&\qquad + \gamma\intG{|U\!\!\mid_{\Gamma} - u|^2d\li} + \xi\intGG{|V\!\!\mid_{\Gamma_2} - v|^2d\pi} + \sigma\intGG{|u - v|^2d\li}\biggr]
		\end{aligned}
	\end{equation}
	We denote by $J_i,\, i = 1,2, \ldots, 6$ the last six terms on the right hand side of \eqref{Eq3}. We have
	\begin{equation}\label{Eq4}
		\begin{aligned}
			J_1 &= \alpha\intO{|U - V|^2d\Pi}= \alpha\intO{|\oU - \oV + \delta_U - \delta_V|^2d\Pi}\\
				  &\geq \epsilon_1\alpha\intO{|\oU - \oV|^2d\Pi} - \frac{2\alpha\epsilon_1}{1 - \epsilon_1}\left(\intO{\delta_U^2d\Pi} + \intO{\delta_V^2d\Pi}\right)\\
				  &\geq \epsilon_1\alpha\intO{|\oU - \oV|^2d\Pi} - \frac{2\alpha\epsilon_1}{1 - \epsilon_1}\left(\left\|\frac{\Pi}{\Li}\right\|_{L^{\infty}(\Omega)}\intO{\delta_U^2d\Li} + \intO{\delta_V^2d\Pi}\right)
		\end{aligned}
	\end{equation}
	for all $\epsilon_1\in (0,1)$. Similarly, we get
\begin{align}
%\begin{aligned}%\label{Eq5}
		J_2 &\geq \epsilon_2\beta\intO{|\oU - \oV|^2d\Li} - \frac{2\beta\epsilon_2}{1-\epsilon_2}\left(\intO{\delta_U^2d\Li} + \left\|\frac{\Li}{\Pi}\right\|_{L^{\infty}(\Omega)}\intO{\delta_V^2d\Pi}\right),\nonumber\\
	%\end{equation}
	%\begin{equation}
	%\label{Eq6}
		J_3& \geq \epsilon_3\lambda\intG{|\oU - \ou|^2dS_{\Li}} - \frac{2\lambda\epsilon_3}{1-\epsilon_3}\left(\intG{\left(\delta_U\!\!\mid_{\Gamma}\right)^2dS_{\Li}} + \left\|\frac{\Li}{\li}\right\|_{L^{\infty}(\Gamma)}\intG{\delta_u^2d\li}\right),\nonumber \\
	%\end{equation}
	%\begin{equation}
	%\label{Eq7}
		J_4 &\geq \epsilon_4\gamma\intG{|\oU - \ou|^2d\li} - \frac{2\gamma\epsilon_4}{1 - \epsilon_4}\left(\left\|\frac{\li}{\Li}\right\|_{L^{\infty}(\Gamma)}\intG{\left(\delta_U\!\!\mid_{\Gamma}\right)^2dS_{\Li}} + \intG{\delta_u^2d\li}\right),\nonumber \\
	%\end{equation}
	%\begin{equation}
%	\label{Eq8}
%		J_5 &\geq \epsilon_5\xi\intGG{|\oV-\ov|^2d\pi} - \frac{2\xi\epsilon_5}{1-\epsilon_5}\left(\left\|\frac{\pi}{\Pi}\right\|_{L^{\infty}(\Gamma_2)}\intGG{\left(\delta_V\!\!\mid_{\Gamma_2}\right)^2dS_{\Pi}} + \intGG{\delta_v^2d\pi}\right),
%\end{equation}
%	and
%	\begin{equation}\label{Eq9}
		J_6 &\geq \epsilon_6\sigma\intGG{|\ou - \ov|^2d\li} - \frac{2\sigma\epsilon_6}{1 - \epsilon_6}\left(\intGG{\delta_u^2d\li} + \left\|\frac{\li}{\pi}\right\|_{L^{\infty}(\Gamma_2)}\intGG{\delta_v^2d\pi}\right),\label{Eq9}
\end{align}
%\end{equation}
with $\epsilon_2,\epsilon_3,\epsilon_4, \epsilon_6 \in (0,1)$.
For $J_5$, the lack of continuity of $\Pi$ at the boundary $\Gamma_2$ prevents a similar estimate as above, since it is unclear how to control the term $\|\frac{\pi}{\Pi}\|_{L^{\infty}(\Gamma_2)}$. However, the weak reversibility of 
system \eqref{Lgl_system} allows first to estimate $J_5 \geq 0$ and then use the triangle inequality to have
	\begin{multline}\label{omega}
		%\begin{gathered}
			\frac 12\left(\epsilon_2\beta\intO{|\oU - \oV|^2d\Li} +  \epsilon_3\lambda\intG{|\oU - \ou|^2dS_{\Li}} + \epsilon_6\sigma\intGG{|\ou - \ov|^2d\li}\right)\\
			\geq \frac{1}{6}\min\left\{\epsilon_2\beta\intO d\Li; \epsilon_3\lambda\intG dS_{\Li} ; \epsilon_6\sigma\intGG d\li \right\}|\overline{V} - \overline{v}|^2
			=: \omega |\overline{V} - \overline{v}|^2.
		%\end{gathered}
	\end{multline}

	By combining \eqref{Eq3}--\eqref{Eq9} and by choosing $\epsilon_1, \ldots, \epsilon_6$ small enough (for instance in order to ensure that for some $\eta_1 >0$
$$
d_LP_L - \frac{2\alpha\epsilon_1}{1 - \epsilon_1}\left\|\frac{\Pi}{\Li}\right\|_{L^{\infty}(\Omega)}- \frac{2\beta\epsilon_2}{1-\epsilon_2} \ge \eta_1 >0
$$
with $\bigl\|\frac{\Pi}{\Li}\bigr\|_{L^{\infty}(\Omega)}\le \frac{A}{a}$ by Lemma \ref{Lgl_equilibrium}), 
we can estimate $\D(\cc - \cc_{\infty}|\cc_{\infty})$ below as
	\begin{equation}\label{Eq10}
		\begin{aligned}
			\D(\cc-\cc_{\infty}|\cc_{\infty}) 
				&\geq \frac{1}{2}\min\{\epsilon_1,\epsilon_2, \epsilon_3, \epsilon_4, 2\omega, \epsilon_6\}\D(\overline{\cc}-\cc_{\infty}|\cc_{\infty})\\
				&\ \ \ + \eta_1\!\intO{\delta_U^2d\Li} + \eta_2\!\intO{\delta_V^2d\Pi} + \eta_3\!\intG{\delta_u^2d\li} + \eta_4\!\intGG{\delta_v^2d\pi}
		\end{aligned}
	\end{equation}
	where $\omega$ is defined in \eqref{omega}. Hence, by using \eqref{Eq1} and \eqref{Eq2}, we have
\begin{align*}
			\D(\cc-\cc_{\infty}|\cc_{\infty}) &\geq \frac 12 K_0\min\{\epsilon_1,\epsilon_2, \epsilon_3, \epsilon_4, 2\omega, \epsilon_6\}\E(\overline{\cc}-\cc_{\infty}|\cc_{\infty})\\
				 &\quad+\min_{i=1..4}\{\eta_i\}\left(\intO{\delta_U^2d\Li}+\intO{\delta_V^2d\Pi}+\intG{\delta_u^2d\li}+\intGG{\delta_v^2d\pi}\right)\\
				 &\geq {\lambda_0}\,\mathcal E(\cc-\cc_{\infty}|\cc_{\infty})
\end{align*}
	with $\lambda_0 = \frac 12 \min\{2K_0\epsilon_1,2K_0\epsilon_2, 2K_0\epsilon_3, 2K_0\epsilon_4, 4K_0\omega, 2K_0\epsilon_6, \eta_1, \eta_2, \eta_3, \eta_4\}$.
\end{proof}
%  ******************************************************* %
%  ******************************************************* %
\section*{Proof of Theorem \ref{mainresult2}}
In this section, we denote by $\uu = (u_0,u_1,\ldots, u_7)$ and $\vv = (w_0,w_1,\ldots,w_7)$. Moreover, we define 
(in the spirit of the relative entropy \eqref{RelativeEntropy_Recall} of first-order reaction networks)
the relative entropy functional associated to \eqref{JAK_system}: 
\begin{equation}\label{RE} 
   \E(\uu|\vv) = \int\limits_{\Ocyt}\left(\frac{|u_0|^2}{w_0} + \frac{|u_1|^2}{w_1}\right)dx + |\Onuc|\sum_{i=2}^{7}\frac{|u_i|^2}{w_i},
\end{equation}
which dissipates (analog to the entropy dissipation \eqref{EntropyDissipation}) due to the following entropy dissipation functional:
    \begin{align}
      	\D&(\uu|\vv) = -\frac{d}{dt} \E(\uu|\vv) 
	= D\int_{\Ocyt}w_0\Bigl|\nabla\frac{u_0}{w_0}\Bigr|^2dx + D\int_{\Ocyt}w_1\Bigl|\nabla\frac{u_1}{w_1}\Bigr|^2dx\nonumber\\
	    & + \ract\, p_{JAK}\!\!\int\limits_{\partial\Ocyt}\!w_0\biggl[\frac{u_0}{w_0}\biggr\vert_{\partial\Ocyt}\! \!- \frac{u_1}{w_1}\biggr\vert_{\partial\Ocyt}\biggr]^2\!dS + \rimpp\!\!\int\limits_{\partial\Onuc}\!w_1\biggl[\frac{u_1}{w_1}\biggr\vert_{\partial\Onuc} \!\!- \frac{u_3}{w_3}\biggr]^2\!dS\nonumber\\
	    & + \rimp\int_{\partial\Onuc}w_0\biggl[\frac{u_0}{w_0}\biggr\vert_{\partial\Onuc} \!- \frac{u_2}{w_2}\biggr]^2dS + \rexp w_2\int_{\partial\Onuc}\biggl[\frac{u_0}{w_0}\biggr\vert_{\partial\Onuc} \!- \frac{u_2}{w_2}\biggr]^2dS\nonumber\\
	    & + \rdelay\,w_7\biggl[\frac{u_7}{w_7} - \frac{u_2}{w_2}\biggr]^2 + \sum_{i=3}^{6}\rdelay\,w_{i}\biggl[\frac{u_i}{w_{i}} - \frac{u_{i+1}}{w_{i+1}}\biggr]^2.\label{EP}
    \end{align}

\begin{lemma}[Existence of a unique positive equilibrium of \eqref{JAK_system}]\label{STAT_equilibrium}\hfill\\
	For any positive initial mass $M>0$, system \eqref{JAK_system} possesses a unique equilibrium $\uu_{\infty} = (\ukhong,\ldots, \ubay)$ satisfying the mass conservation \eqref{MassInf}, i.e.
	\begin{equation*} %\label{MassInf} 
	  \int_{\Ocyt}(\ukhong(x)+\umot(x))dx + |\Onuc|\sum_{i=2}^{7}u_{i,\infty} = M>0.
	\end{equation*}
	Moreover, $\uhai, \ldots, \ubay$ are positive and $\ukhong, \umot\in C(\overline{\Omega}_{\mathrm{cyt}})\cap C^2(\Ocyt)$ satisfy
	\begin{equation*}
		0< b\leq \ukhong(x), \umot(x) \leq B <+\infty, \qquad\text{ for all } x\in\Ocyt
	\end{equation*}
	for some constants $0<b\le B\le+\infty$.	
\end{lemma}
\begin{proof}
	From \eqref{ODEs}, we easily see that 
\begin{align}\label{y1}
	\uba &= \ubon = \unam = \usau = \ubay = \frac{\rimpp}{\rdelay|\partial\Onuc|}\int_{\partial\Onuc}\umot(y) dS,\\
	%\end{equation}
%	and
%	\begin{equation}
%		\begin{aligned}
\uhai &= \frac{\rdelay}{\rexp}\ubay + \frac{\rimp}{\rexp|\partial\Onuc|}\int_{\partial\Onuc}\ukhong(y)dS\nonumber\\
&= \frac{\rimpp}{\rexp|\partial\Onuc|}\int_{\partial\Onuc}\umot(y)dS +  \frac{\rimp}{\rexp|\partial\Onuc|}\int_{\partial\Onuc}\ukhong(y)dS.\label{y2}
\end{align}
It thus remains to solve the following non-local elliptic system for $\ukhong$ and $\umot$,
	\begin{subequations}
		\begin{equation}\label{equi_0}
			\begin{cases}
				D\Delta\ukhong(x) = 0, &\quad x\in\Ocyt,\\
				D\partial_{n_1}\ukhong(y) = -\frac{\ract}{|\partial\Ocyt|}p_{JAK}\ukhong(y), &\quad y\in\partial\Ocyt,\\
				D\partial_{n_2}\ukhong(y)=-\frac{\rimp}{|\partial\Onuc|}\ukhong(y) \\ \quad  +\frac{1}{|\partial\Onuc|}\biggl(\frac{\rimpp}{|\partial\Onuc|}\int\limits_{\partial\Onuc}\umot\, dS +  \frac{\rimp}{|\partial\Onuc|}\int\limits_{\partial\Onuc}\ukhong\, dS\biggr), &\quad y\in\partial\Onuc,
			\end{cases}
		\end{equation}
		%and
		\begin{equation}\label{equi_1}
			\begin{cases}
				D\Delta\umot(x) = 0, &x\in\Ocyt,\\
				D\partial_{n_1}\umot(y) = \frac{\ract}{|\partial\Ocyt|}p_{JAK}\ukhong(y),&y\in\partial\Ocyt,\\
				D\partial_{n_2}\umot(y) = -\frac{\rimp2}{|\partial\Onuc|}\umot(y), &y\in\partial\Onuc.
			\end{cases}
		\end{equation}
	\end{subequations}
	subject to the constraint, which follows from the mass conservation, \eqref{y1} and \eqref{y2},
	\begin{multline}\label{reduce_mass}
		\int_{\Ocyt}(\ukhong + \umot)dx + \frac{\rimp}{\rexp |\partial\Onuc|}\int_{\partial\Onuc}\ukhong dS \\+  \left(5\frac{\rimpp}{\rdelay|\partial\Onuc|}\frac{\rimpp}{\rexp |\partial\Onuc|}\right)\int_{\partial\Onuc}\umot dS = M.
	\end{multline}
	By considering an auxiliary system as follows %\textcolor{red}{unclear to me what is $u_0$ and what $u_{0,\infty}$.}
	\begin{equation}\label{JAK_auxiliary}
		\begin{cases}
			D\Delta u_0 = 0, &x\in\Ocyt,\\
			D\partial_{n_1}u_0 + \frac{\ract}{|\partial\Ocyt|}p_{JAK}u_0 = 0, &y\in \partial\Ocyt,\\
			D\partial_{n_2}u_0+\frac{\rimp}{|\partial\Onuc|}u_0 =\frac{1}{|\partial\Onuc|}\Bigl(\frac{\rimpp}{|\partial\Onuc|}\int\limits_{\partial\Onuc}\!\!\!\widehat{u_1}\, dS +  \frac{\rimp}{|\partial\Onuc|}\int\limits_{\partial\Onuc}\!\!\!\widehat{u_0}\, dS\Bigr), &y\in\partial\Onuc,\\
			D\Delta u_1 = 0, &x\in\Ocyt,\\
			D\partial_{n_1}u_1 = \frac{\ract}{|\partial\Ocyt|}p_{JAK}\widehat{u_0}(y),&y\in\partial\Ocyt,\\
			D\partial_{n_2}u_1 +\frac{\rimp2}{|\partial\Onuc|}u_1 = 0, &y\in\partial\Onuc
		\end{cases}
	\end{equation}
	we can use similar arguments in Lemma \ref{Lgl_equilibrium}, namely define for non-negative $(\widehat{u_0}, \widehat{u_1})$ an operator $\mathfrak L: (\widehat{u_0}, \widehat{u_1}) \mapsto (u_0, u_1)$ as the solution to \eqref{JAK_auxiliary}. Then, the existence of a nonnegative weak solution $(\ukhong, \umot)\in H^1(\Ocyt)\times H^1(\Ocyt)$ to \eqref{equi_0}-\eqref{equi_1} follows from the Schauder fixed point theorem applied to $\mathfrak L$. In return, this implies nonnegative equilibria $(\ukhong, \ldots, \ubay)$ to system \eqref{JAK_system}.
	
	By applying standard bootstrap arguments to \eqref{equi_0} and \eqref{equi_1}, we obtain that $(\ukhong, \umot)$ is in fact a classical solution, namely $\ukhong, \umot \in C(\overline{\Omega}_{\mathrm{cyt}})\cap C^2(\Ocyt)$. Hence, $\ukhong$ and $\umot$ are uniformly bounded above by a positive constant thanks to the compactness of $\overline{\Omega}_{\mathrm{cyt}}$. We now prove the strict positivity of $\ukhong$ and $\umot$. First we show that $\ukhong$ is not identical to zero on $\partial\Ocyt$. Indeed, assume the reserve is true, we then obtain from \eqref{equi_1} that $\umot\equiv 0$. It follows then from \eqref{equi_0} and the strong maximum principle that $\ukhong\equiv 0$. This violates the mass conservation \eqref{reduce_mass}. Therefore, $\ukhong\geq 0$ is not identical to zero on $\partial\Ocyt$ and we get from \eqref{equi_1} and the maximum principle that $\umot(x)>0$ for all $x\in\overline{\Omega}_{\mathrm{cyt}}$ which leads to a strictly positive lower bound of $\umot$. The strict positivity of $\ukhong$ also follows from $\ukhong \in C(\overline{\Omega}_{\mathrm{cyt}})$ and strong maximum principle since the second boundary condition in \eqref{equi_0} implies
	\begin{equation*}
		D\partial_n\ukhong(y) + \frac{\rimp}{|\partial\Onuc|}\ukhong(y) \geq \frac{\rimpp}{|\partial\Onuc|^2}\int_{\partial\Onuc}\umot dS
	\end{equation*}
	and $\umot$ is strictly positive.
%	
%	\vskip 1cm
%	Moreover, since $\ukhong, \umot$ are nonnegative and not identically zero, we have
%	\begin{equation*}
%		\ukhong(x), \umot(x) >0 \qquad x\in\overline{\Omega}_{\mathrm{cyt}}
%	\end{equation*}
%	(\textcolor{blue}{needs references about maximum principle for nonlocal problems!}) which implies
%	\begin{equation*}
%		\ukhong(x), \umot(x) \geq b >0 \qquad x\in\overline{\Omega}_{\mathrm{cyt}}
%	\end{equation*}
%	for some $b>0$. The upper bounds of $\ukhong$ and $\umot$ are obvious due to the compactness of $\overline{\Omega}_{\mathrm{cyt}}$ and the continuity of $\umot$ and $\ukhong$.
	The uniqueness of the equilibrium can be proved similar to Lemma \ref{Lgl_equilibrium} thanks to the entropy structures \eqref{RE} and \eqref{EP}. We omit here the proof.
\end{proof}

As in the previous section, we %intend to improve the presentation of the results by 
introduce the following short notations
\begin{itemize}
  \item {New variables:}\quad
  %\begin{equation*}
    $v_i = \frac{u_i - u_{i,\infty}}{u_{i,\infty}}, \ \text{ for } i=0,1,\ldots,7$.
  %\end{equation*}
  \item {New measures:}\quad
  %\begin{equation*}
    $d\ukhong = \ukhong dx,\quad d\umot = \umot dx, \quad d\sigma_{\ukhong} = \ukhong\!\!\mid_{\partial\Ocyt}dS$,
  %\end{equation*}
  \begin{equation*}
d\sigma_{\umot} = \umot\!\!\mid_{\partial\Ocyt}dS, \quad
dS_{\ukhong} = \ukhong\!\!\mid_{\partial\Onuc}dS,\quad dS_{\umot} = \umot\!\!\mid_{\partial\Onuc}dS.
  \end{equation*}
  \item {New parameters:}\
  %\begin{equation*}
    $\alpha = \ract, \ \beta = p_{JAK},\  \gamma = \rimp, \quad \sigma = \rimpp, \ \kappa = \rexp, \ \xi = \rdelay$.
  %\end{equation*}
  \item {Weighted averages:}
  \begin{equation*}
    \overline{v}_0 = \frac{1}{\int_{\Onuc}d\ukhong}\int_{\Onuc}v_0(t,x)d\ukhong,\qquad \overline{v}_1 = \frac{1}{\int_{\Onuc}d\umot}\int_{\Onuc}v_1(t,x)d\umot.
  \end{equation*}
\end{itemize}
Then, analog to \eqref{entropyshift}--\eqref{EntropyDissipationR}, we can rewrite the relative entropy 
\begin{equation}\label{reRE} 
  \E(\uu-\uinf|\uinf) = \int_{\Ocyt}|v_0|^2d\ukhong + \int_{\Ocyt}|v_1|^2d\umot + |\Onuc|\sum_{i=2}^{7}u_{i,\infty}|v_i|^2,
\end{equation}
and its entropy dissipation functional $\D(\uu - \uinf|\uinf) = -\frac{d}{dt} \E(\uu-\uinf|\uinf)$ as 
  \begin{align}
  &\D(\uu - \uinf|\uinf) %= -\frac{d}{dt} \E(\uu-\uinf|\uinf)      
        = D\int_{\Ocyt}\left|\nabla v_0\right|^2d\ukhong + D\int_{\Ocyt}\left|\nabla v_1\right|^2d\umot
        + \sum_{i=3}^{6}\xi u_{i,\infty}\left[v_i - v_{i+1}\right]^2 \nonumber\\
	    & + \alpha\beta\int_{\partial\Ocyt}\left[v_0\!\!\mid_{\partial\Ocyt} - v_1\!\!\mid_{\partial\Ocyt}\right]^2d\sigma_{\ukhong}+\sigma\int_{\partial\Onuc}\left[v_1\!\!\mid_{\partial\Onuc} - v_3\right]^2dS_{\umot}\nonumber\\
	    & + \gamma\int_{\partial\Onuc}\left[v_0\!\!\mid_{\partial\Onuc} - v_2\right]^2dS_{\ukhong} + \kappa\uhai\int_{\partial\Onuc}[v_0\!\!\mid_{\partial\Onuc} - v_2]^2dS + \xi\ubay[v_7 - v_2]^2.\label{reEP}
    \end{align}
Finally, the mass conservation law rewrite as
\begin{equation}\label{reMassConservation}
	\overline{v}_0\int_{\Ocyt}d\ukhong + \overline{v}_1\int_{\Ocyt}d\umot + |\Onuc|\sum_{i=2}^{7}u_{i,\infty}v_i = 0.
\end{equation}

\begin{proof}[Proof of Theorem \ref{mainresult2}]\hfill\\
The proof of Theorem \ref{mainresult2} is a direct consequence of the functional inequality \eqref{desired_JAK} in the following  Lemma \ref{EEDE_JAK}  
and a Gronwall argument applied to the entropy dissipation law \eqref{reEP}, which is weakly satisfied by the global solutions constructed in \cite{FNR13}.
\end{proof}

%The result of the following lemma, combining with the classical Gronwall lemma implies immediately Theorem \ref{mainresult2}.
\begin{lemma}[Entropy entropy-dissipation estimate for system \eqref{JAK_system}] \label{EEDE_JAK}\hfill\\
Fix a positive initial mass $M>0$. Then, for any measurable state $\uu = (u_0, \ldots, u_7)$ satisfying the mass conservation
%\begin{equation*}
	$\int_{\Ocyt}(u_0(x) + u_1(x))dx + |\Onuc|\sum_{i=2}^{7}u_i = M$
%\end{equation*}
the entropy entropy-dissipation inequality
\begin{equation}\label{desired_JAK}
	\D(\uu - \uinf|\uinf) \geq \lambda_1\,\E(\uu - \uinf|\uinf)
\end{equation}
holds where $\uinf$ is as in Lemma \ref{STAT_equilibrium} and the constant $\lambda_1>0$ can be estimated explicitly.
\end{lemma}
\begin{proof}
	The proof of this lemma is similar to that of Lemma \ref{EEDE_Lgl}. First, it is straightforward to verify that the relative entropy satisfies to following additivity property:
	\begin{equation}\label{additive}
		\E(\uu - \uinf|\uinf) = \E(\uu - \overline{\uu}|\uinf) + \E(\overline{\uu} - \uinf|\uinf)
	\end{equation}
	where
	\begin{eqnarray*}
		&\E(\uu - \overline{\uu}|\uinf) = \int_{\Ocyt}|v_0 - \overline{v}_0|^2d\ukhong + \int_{\Ocyt}|v_1 - \overline{v}_1|^2d\umot,\\
	%\end{equation*}
	%and
	%\begin{equation*}
		&\E(\overline{\uu} - \uinf|\uinf) = \overline{v}_0^2\int_{\Ocyt}d\ukhong + \overline{v}_1^2\int_{\Ocyt}d\umot + |\Onuc|\sum_{i=2}^{7}u_{i,\infty}|v_i|^2.
	\end{eqnarray*}
	
	%\smallskip
	\noindent{\bf Step 1.} Thanks to the mass conservation law \eqref{reMassConservation}, it follows from Lemma \ref{finite_dimension} and the weak reversibility of the reaction network Fig. \ref{JSFig} that 	for an explicit constant $L_0>0$
	\begin{equation}\label{e3}
		\D(\overline{\uu} - \uinf|\uinf) \geq L_0\,\E(\overline{\uu} - \uinf|\uinf).
	\end{equation}

	%\smallskip
	\noindent{\bf Step 2.} The term $\E(\uu - \overline{\uu}|\uinf)$ can be controlled in terms of the entropy dissipation by using the weighted Poincar\'{e} inequalities as follow
	\begin{equation}\label{e4}
		\begin{aligned}
			\frac 12 \D(&\uu-\uinf|\uinf) \geq D\int_{\Ocyt}|\nabla v_0|^2d\ukhong + D\int_{\Ocyt}|\nabla v_1|^2d\umot\\
			&\geq L_1\Bigl(\int_{\Ocyt}|v_0 - \overline{v}_0|^2d\ukhong + \int_{\Ocyt}|v_1 - \overline{v}_1|^2d\umot\Bigr) = L_1\,\E(\uu - \overline{\uu}|\uinf),
		\end{aligned}
	\end{equation}
	where $L_1=D \min\{P_{\ukhong},P_{\umot}\}$ depends on $D$ and the corresponding weighted Poincar\'e constants.
Next, we prove for some constant $L_2>0$ that
	\begin{equation}\label{e5}
		\frac 12\D(\uu - \uinf|\uinf) \geq L_2\,\D(\overline{\uu} - \uinf|\uinf),
	\end{equation}
which yields the desired estimate \eqref{desired_JAK} with $\lambda_1 = \min\{L_1, L_0L_2\}$ by combining \eqref{additive}, \eqref{e3}, \eqref{e4} and \eqref{e5}. To prove \eqref{e5}, we will use the $\mu$-weighted trace inequalities 
	\begin{equation*}
		\int_{\Ocyt}|\nabla f|^2d\mu \geq T_{\sigma_\mu}\int_{\partial\Ocyt}|f - \overline{f}|^2\,d\sigma_{\mu}, \quad  \int_{\Onuc}|\nabla f|^2d\mu \geq T_{S_\mu}\int_{\partial\Onuc}|f - \overline{f}|^2\,dS_{\mu}, 
	\end{equation*}
	where $\overline{f} = \frac{1}{\int_{\Ocyt}d\mu}\int_{\Ocyt}fd\mu$ to estimate, using the triangle inequality,
	\begin{equation*} %\label{e3}
		\begin{aligned}
			&\D(\uu - \uinf|\uinf) \geq DT_{\mu}\!\!\!\!\!\int\limits_{\partial\Ocyt}\!\!\!|\!v_0 - \overline{v}_0|^2d\sigma_{\ukhong} +DT_{\mu}\!\!\!\!\!\int\limits_{\partial\Ocyt}\!\!\!\!|v_1 - \overline{v}_1|^2d\sigma_{\umot}+ \xi\ubay|v_7 - v_2|^2  \\
			&\ \  + \alpha\beta\! \int_{\partial\Ocyt}\!\!|v_0 - v_1|^2d\sigma_{\ukhong}+ DT_{\mu}\!\int_{\partial\Onuc}\!\!|v_0 - \overline{v}_0|^2dS_{\ukhong} + \gamma\!\int_{\partial\Onuc}\!\!|v_0 - v_2|^2dS_{\ukhong}\\
		    &\ \  + DT_{\mu}\int_{\partial\Onuc}|v_1 - \overline{v}_1|^2dS_{\umot} + \sigma\int_{\partial\Onuc}|v_1 - v_3|^2dS_{\umot} + \sum_{i=3}^{6}\xi u_{i,\infty}\left[v_i - v_{i+1}\right]^2\\
		    &\geq C_1|\overline{v}_0 - \overline{v}_1|^2\int_{\partial\Ocyt}d\sigma_{\ukhong} + C_2|\overline{v}_0 - v_2|^2\int_{\partial\Onuc}dS_{\ukhong} + C_3|\overline{v}_1 - v_3|^2\int_{\partial\Onuc}dS_{\umot}\\
		    &\ \ + \sum_{i=3}^{6}\xi u_{i,\infty}|v_i - v_{i+1}|^2
		    \geq 2L_2\,\D(\overline{\uu} - \uinf|\uinf)
		\end{aligned}
	\end{equation*}	
%	\textcolor{red}{Computation for $C_1$}
where $T_{\mu}=\min\{T_{\sigma_{\ukhong}}, T_{\sigma_{\umot}}, T_{S_{\ukhong}}, T_{S_{\umot}}\}$,
%	\textcolor{red}{To DO: estiamte like \eqref{Eq4}}
%	\begin{equation*}\label{Constants1}
%		\begin{gathered}
	$		C_1 = \frac{DT_{\mu}}{3}\min\Bigl\{1;\alpha\beta;\min\limits_{y\in\partial\Ocyt}\frac{\umot(y)}{\ukhong(y)}\Bigr\}$, 
	$		C_2 = \frac{\min\{DT_{\mu};\gamma\}}{2}$,  $C_3 = \frac{\min\{DT_{\mu};\sigma\}}{2}$
%		\end{gathered}
%	\end{equation*}	
	and
	%\begin{equation*}
$L_2 = \frac 12 \min\Bigl\{\frac{C_1}{\alpha\beta}; \frac{C_2}{2\gamma}; \frac{C_2\int_{\partial\Onuc}dS_{\ukhong}}{2\kappa \uhai \int_{\partial\Onuc}dS}; \frac{C_3}{\sigma}; 1\Bigr\}$
	%\end{equation*}
	which can be computed explicitly, for instance
	\begin{equation*}
		\begin{aligned}
			&DT_{\mu}\!\int_{\partial\Ocyt}|v_0 - \overline{v}_0|^2d\sigma_{\ukhong} +DT_{\mu}\!\int_{\partial\Ocyt}|v_1 - \overline{v}_1|^2d\sigma_{\umot} + \alpha\beta\! \int_{\partial\Ocyt}|v_0 - v_1|^2d\sigma_{\ukhong}\\
%			&DT_{\mu}\int\limits_{\partial\Ocyt}|v_0 - \overline{v}_0|^2d\sigma_{\ukhong} +DT_{\mu}\min_{y\in\partial\Ocyt}\frac{\umot(y)}{\ukhong(y)}\int\limits_{\partial\Ocyt}|v_1 - \overline{v}_1|^2d\sigma_{\ukhong} + \alpha\beta \int\limits_{\partial\Ocyt}|v_0 - v_1|^2d\sigma_{\ukhong}\\
			&\geq \frac{1}{3}DT_{\mu}\min\left\{1;\; \min_{y\in\partial\Ocyt}\frac{\umot(y)}{\ukhong(y)};\;\alpha\beta\right\}\int_{\partial\Ocyt}(v_0 - \overline{v}_0 + \overline{v}_1 - v_1 - v_0 + v_1)^2d\sigma_{\ukhong}\\
			&= C_1|\overline{v}_0 - \overline{v}_1|^2\int_{\partial\Ocyt}d\sigma_{\ukhong}.
		\end{aligned}
	\end{equation*}
\end{proof}

\vskip 0.5cm
\noindent{\bf Acknowledgements.} The second author was supported by International Research Training Group IGDK 1754. This work has partially been supported by NAWI Graz.

\end{document}